\documentclass[11pt,a4paper,twoside]{article}
\usepackage[T1]{fontenc}
\usepackage{lmodern,amsmath,amssymb,amsthm,mathtools,booktabs}
\usepackage[left=29mm,right=29mm,top=29mm,bottom=29mm,headheight=13pt,headsep=20pt,footskip=30pt]{geometry}
\usepackage{fancyhdr,titlesec,microtype,etoolbox}
\usepackage[hidelinks]{hyperref}
\usepackage{authblk}

\titleformat{\section}{\normalfont\large\bfseries\filcenter}{\thesection.}{0.5em}{}
\titleformat{\subsection}{\normalfont\normalsize\bfseries}{\thesubsection}{0.5em}{}
\titlespacing*{\section}{0pt}{14pt plus 3pt minus 2pt}{9pt}
\titlespacing*{\subsection}{0pt}{12pt plus 3pt minus 2pt}{7pt}
\theoremstyle{plain}
\newtheorem{theorem}{Theorem}[section]
\newtheorem{lemma}[theorem]{Lemma}
\newtheorem{corollary}[theorem]{Corollary}
\newtheorem{proposition}[theorem]{Proposition}

\theoremstyle{remark}
\newtheorem{remark}[theorem]{Remark}

\numberwithin{equation}{section}
\newcommand{\Ham}{\operatorname{Ham}}
\newcommand{\Aut}{\operatorname{Aut}}
\newcommand{\Mod}{\operatorname{Mod}}
\newcommand{\PSL}{\operatorname{PSL}}
\newcommand{\Sym}{\operatorname{Sym}}
\newcommand{\norm}[1]{\left\lVert #1\right\rVert}

\newcommand{\id}{\operatorname{id}}

\appto\normalsize{%
\setlength{\abovedisplayskip}{6pt plus 2pt minus 4pt}%
\setlength{\belowdisplayskip}{6pt plus 2pt minus 4pt}%
\setlength{\abovedisplayshortskip}{2pt plus 2pt}%
\setlength{\belowdisplayshortskip}{5pt plus 2pt minus 3pt}}

\begin{document}

\thispagestyle{plain}
\begin{center}
{\Large\bfseries The space of equators has infinite Hofer's diameter\par}

\title{The space of equators has infinite Hofer's diameter}




\vspace{6pt}
{\large Leonid Polterovich\footnote{Tel Aviv University, 
School of Mathematical Sciences,
69978, Tel Aviv, Israel,
polterov@tauex.tau.ac.il} and Egor Shelukhin\footnote{Egor Shelukhin, 
Department of Mathematics and Statistics,
University of Montreal,
C.P. 6128 Succ.  Centre-Ville
Montreal, QC, H3C 3J7, Canada,
egor.shelukhin@umontreal.ca} \par}
\end{center}
\vspace{4pt}

\begin{abstract}
We settle several questions about Hofer's geometry of the two-sphere. We prove that the space of equators
has infinite Hofer diameter, construct simple commutators of arbitrarily
large Hofer norm, and exhibit a Hamiltonian diffeomorphism whose iterates
escape to infinity in Hofer distance from the set of autonomous
diffeomorphisms. Finally, we construct a bi-Lipschitz group monomorphism
from the free group of rank two, equipped with its conjugation-invariant
generator norm, into the Hamiltonian group equipped with Hofer's metric.
The proof rests on an interplay between Lagrangian Floer theory in the fourth
symmetric product of the sphere and combinatorial group theory. This paper was written with substantial assistance from AI.
\end{abstract}

\begingroup
\tableofcontents
\endgroup

\section{Introduction and main results}
\subsection{Hofer's metric and the main results}
Equip the two dimensional sphere $S^2$ with an area form $\omega$ of total area one. Write
$\Ham(S^2)$ for its group of Hamiltonian diffeomorphisms.
 The Hofer length of a smooth path $\{\phi_H^t\}$ of Hamiltonian diffeomorphisms generated by a
Hamiltonian $H:[0,1]\times S^2\to\mathbb R$ is
\[
\norm{H}_{(1,\infty)}=\int_0^1\bigl(\max H_t-\min H_t\bigr)\,dt.
\]
The Hofer norm and its bi-invariant metric are
\[
\norm{\phi}_{\rm H}=\inf_{\phi_H^1=\phi}\norm{H}_{(1,\infty)},\qquad
d_{\rm H}(\phi,\psi)=\norm{\psi^{-1}\phi}_{\rm H}.
\]
Here and below, an infimum over Hamiltonians is taken over all smooth
Hamiltonian paths with the prescribed endpoint. See \cite{Hof90, Pol01}
for background. Throughout the paper we use the Hamiltonian convention
\[
{\iota_{X_{H_t}}\omega=-dH_t.}
\]
We use the notation $[g,h]=ghg^{-1}h^{-1}$, and compositions of maps act from right to left.

An \emph{equator} is a smooth embedded circle whose complementary disks
both have area $1/2$. The Hamiltonian group acts transitively on the
space $\mathcal E$ of unoriented equators. Its quotient Hofer metric is
\begin{equation}\label{eq:equator-metric}
d_{\mathcal E}(E,E')=\inf\{\norm{\phi}_{\rm H}:\phi(E)=E'\}.
\end{equation}
The autonomous maps form the subset
$\Aut(S^2)=\{\phi_H^1:H\in C^\infty(S^2)\}$.

\begin{theorem}\label{thm:equators}
The space of equators has infinite Hofer diameter. In fact there is a Hamiltonian diffeomorphism $f\in\Ham(S^2)$, an equator $E_0$, and constants
$c,C>0$ such that
\begin{equation}\label{eq-est-eq}
c|m-n|\le d_{\mathcal E}(f^m(E_0),f^n(E_0))\le C|m-n|
\qquad(m,n\in\mathbb Z).\end{equation}

\end{theorem}

\begin{remark}\label{rem-equators-quant}
We shall give two proofs of this statement. One of them, presented in Section \ref{sec:direct-equators} below, is direct. The other will follow quickly from Theorem \ref{thm:commutators} below.
\end{remark}


\begin{theorem}\label{thm:commutators}
There are simple commutators of arbitrarily large Hofer norm in
$\Ham(S^2)$. In fact, there exist $A,B\in\Ham(S^2)$ and $c>0$ such that
\[
\norm{[A^m,B^m]}_{\rm H}\ge cm
\qquad\text{for every integer }m\ge1.
\]
\end{theorem}


\begin{theorem}\label{thm:autonomous}
There exists a sequence of Hamiltonian diffeomorphisms $f_n \in \Ham(S^2)$
with linearly
growing Hofer distance from all autonomous maps:
\[
cn\le d_{\rm H}(f_n,\Aut(S^2))\le Cn\qquad(n\ge1)
\]
for some $c,C>0$.
In fact, $f_n$ can be taken as iterates $f^n$ of some $f\in\Ham(S^2)$.
\end{theorem}

Our final result deals with quasi-isometric embeddings of the free group
to $\Ham(S^2)$. Let $F_2=\langle x,y\rangle$. Its conjugation-invariant generator norm is
\[
\nu_F(w)=\min\left\{r:w=\prod_{i=1}^r c_i e_i c_i^{-1},\quad
c_i\in F_2,\quad e_i\in\{x^{\pm1},y^{\pm1}\}\right\}.
\]
The empty product has length zero. Conjugating elements are unrestricted
within $F_2$ and do not contribute to the length. This gives the
bi-invariant metric $d_F(v,w)=\nu_F(v^{-1}w)$; it is not ordinary
reduced-word distance.

\begin{theorem}\label{thm:f2embedding} There exists a bi-Lipschitz group monomorphism $\Phi:F_2\to\Ham(S^2)$  satisfying
\[
\frac1{39936}\nu_F(v^{-1}w)
\le d_{\rm H}\bigl(\Phi(v),\Phi(w)\bigr)
\le4\nu_F(v^{-1}w)\qquad(v,w\in F_2).
\]
\end{theorem}

\paragraph{Historical and bibliographical comments:}
The question of whether the space of equators has infinite Hofer diameter
has been part of the folklore since 1991 and appears as Problem~32 in the
third edition of McDuff--Salamon~\cite{MS17}.
Important earlier work is due to Khanevsky~\cite{Kha09,Kha11}, who proved
unboundedness for diameters of the disk and related spaces of curves on
the annulus, and developed the method of using Hofer-Lipschitz
quasimorphisms that vanish on stabilizers. Recent work by Serraille--Trifa~\cite{ST26} studied a potential
application of this method to the question about equators. Finally, analogues of the equator problem for other metrics have been
resolved by Jia--Webb~\cite{JW25} for quantitative fragmentation metrics
and by Brandenbursky--Shelukhin~\cite{BS26} for hydrodynamical $L^p$
metrics.

The study of unbounded Hofer distance from autonomous maps goes back to
Polterovich--Shelukhin~\cite{PS16}, who treated closed surfaces of genus
at least four; the general question is Problem~22 in~\cite{MS17}.
Chor~\cite{Cho24} extended the result to genera two and three, and
Khanevsky~\cite{Kha22b} to the torus; Grunder~\cite{Gru26} subsequently
proved a stronger torus result concerning distance from maps of bounded
entropy. Higher-dimensional examples appear in
Polterovich--Shelukhin~\cite{PS16},
Polterovich--Shelukhin--Stojisavljevi\'c~\cite{PSS17},
Jun Zhang~\cite{Zha19}, and Zhijing Wendy Wang~\cite{Wan26}.
For related obstructions involving curves and integrable maps, see
Khanevsky~\cite{Kha21,Kha22a}; for the corresponding question in
hydrodynamical $L^p$ metrics, see Brandenbursky--Shelukhin~\cite{BS15}.

The question of unbounded Hofer norm of simple commutators was posed by
McDuff and Polterovich in~\cite[Question~4]{McD}.
Khanevsky~\cite{Kha16} answered it for surfaces of positive genus and
certain products; Theorem~\ref{thm:commutators} treats the sphere.

For surfaces of genus at least two and certain other manifolds, monomorphisms of free groups into the 
asymptotic cone of the Hamiltonian group (see Section \ref{sec:cone}) were constructed in~\cite{AGK19},~\cite{Cho24},~\cite{Wan26}.
See also \cite{BS17, BMS22} for results and a review of the literature on quasi-isometric group embeddings in the setting of $L^p$ metrics.

Braid and linking methods in Hofer geometry have been developed by
Alves--Meiwes~\cite{AM24}, Connery-Grigg~\cite{CG24}, and
Grunder~\cite{Gru26}. Closely related applications of link Floer homology
to braid types and Hofer energy appear in Morabito~\cite{Mor25},
Chen~\cite{Che24}, Morabito--Trifa~\cite{MT25}, and Trifa~\cite{Tri24}.
These should be viewed alongside the earlier geometric arguments in
Khanevsky's thesis~\cite{Kha11} and the periodic-orbit and Floer-theoretic
constructions used to estimate distance from autonomous maps, including
those of Zhang~\cite{Zha19} and Wang~\cite{Wan26}.

\paragraph{Organization of the paper:}
In Section \ref{sec:lemmas} we present reductions
\[
\text{Theorem~\ref{thm:f2embedding}}
\ \Longrightarrow\ \text{Theorem~\ref{thm:commutators}}
\ \Longrightarrow\ \text{Theorem~\ref{thm:equators}}
\ \Longrightarrow\ \text{Theorem~\ref{thm:autonomous}}.
\]
In Section \ref{sec:floer} we present the input from Floer theory, including the fundamental Strip Lemma \ref{lem:strip}. This lemma eventually leads to a link between Hofer's geometry and algebra, see Lemma \ref{lem:meridian} in Section \ref{sec:braid-topology}. It is used for the proof of Theorem \ref{thm:f2embedding} in Section \ref{sec:f2embedding}, and for the more direct proof of
Theorem~\ref{thm:equators} in Section~\ref{sec:direct-equators}. In this proof the algebraic input is provided by quasi-morphisms. Finally in Section \ref{sec:misc} we provide a few further related results.

\paragraph{Acknowledgements and AI declaration.}
The statements of the main theorems and the reductions were
already known to us from our many years of work on this topic.
The idea of using iterates of a pseudo-Anosov map to prove
that the space of equators has infinite diameter and that
the distance to autonomous maps is unbounded was motivated
in part by Webb's talk on his work~\cite{JW25} at the Geometry
and Dynamics Seminar at Tel Aviv University.
The Floer-theoretic techniques we used were developed
in earlier papers~\cite{CHMSS22,PS23}.

We supplied this and other related material to AI through a long sequence
of prompts (each of us used ChatGPT 6 Astra).
The AI eventually came up with the key Strip Lemma
and its algebraic consequences and produced all the proofs,
both those using Floer theory and those using combinatorial
group theory. It remained for us to clean and simplify the proofs
and clarify the exposition.

\section{Reductions}\label{sec:lemmas}

\begin{proof}[Proof of Theorem \ref{thm:commutators} assuming Theorem \ref{thm:f2embedding}]
Let \(F_2=\langle x,y\rangle\), equipped with the conjugation-invariant
generator norm \(\nu_F\). By
Brandenbursky, Gal, K\k{e}dra and Marcinkowski
\cite[Propositions~2.1 and~2.5]{BGKM16}, the norm \(\nu_F(w)\) equals
the least number of letters that must be deleted from any fixed word
representing \(w\) so that the remaining word represents the identity.
Here \cite[Proposition~2.1]{BGKM16} identifies the cancellation norm with the
conjugation-invariant generator norm, and \cite[Proposition~2.5]{BGKM16} implies
that the cancellation length is independent of the representing
word in a free group.

We claim that
\[
    \nu_F([x^m,y^m])=2m
    \qquad (m\geq 1).
\]
Deleting all \(2m\) occurrences of \(x^{\pm1}\) from
\(x^m y^m x^{-m}y^{-m}\) leaves \(y^my^{-m}\), proving the upper bound.
Conversely, any word obtained by deleting letters has the form
\[
    x^{n_1}y^{n_2}x^{-n_3}y^{-n_4},
    \qquad 0\leq n_1,n_2,n_3,n_4\leq m.
\]
If this word represents the identity, abelianization gives
\(n_1=n_3\) and \(n_2=n_4\). If both \(n_1\) and \(n_2\) are positive, the word
\(x^{n_1}y^{n_2}x^{-n_1}y^{-n_2}\) is freely reduced and nonempty.
Thus either \(n_1=n_3=0\) or \(n_2=n_4=0\).
In either case, all \(2m\) occurrences of one of the generators
and its inverse have been deleted. This proves the lower bound
and hence the claim.

Let \(\Phi:F_2\to\operatorname{Ham}(S^2)\) be the homomorphism
provided by Theorem~\ref{thm:f2embedding}, and put \(A=\Phi(x)\) and \(B=\Phi(y)\).
The homomorphism property and the lower bound in Theorem~\ref{thm:f2embedding} give
\[
    \bigl\|[A^m,B^m]\bigr\|_{\rm H}
    =\bigl\|\Phi([x^m,y^m])\bigr\|_{\rm H}
    \geq \frac{\nu_F([x^m,y^m])}{39936}
    =\frac{m}{19968}.
\]
Therefore the single commutators \([A^m,B^m]\) have unbounded
Hofer norm, as asserted.
\end{proof}

\begin{proof}[Proof of Theorem \ref{thm:equators} assuming Theorem \ref{thm:commutators}]
It suffices to show that for every $\phi,\psi\in\Ham(S^2)$ and any equator $E$,
\begin{equation}
\label{eq-estimate}
d_{\rm H}([\phi,\psi],\id)
\leq 4\max\bigl\{d_{\mathcal E}(\phi(E),E),d_{\mathcal E}(\psi(E),E)\bigr\}+6.
\end{equation}
Indeed, take $A,B\in\Ham(S^2)$ and $c>0$ from Theorem~\ref{thm:commutators}, and put
\[
a_m=d_{\mathcal E}(A^m(E),E),\qquad
b_m=d_{\mathcal E}(B^m(E),E)\qquad(m\ge1).
\]
Both sequences are subadditive because the Hamiltonian group acts by isometries.
Fekete's lemma therefore gives
\[
\alpha=\lim_{m\to\infty}\frac{a_m}{m}
      =\inf_{m\ge1}\frac{a_m}{m},\qquad
\beta=\lim_{m\to\infty}\frac{b_m}{m}
     =\inf_{m\ge1}\frac{b_m}{m}.
\]
The commutator lower bound and \eqref{eq-estimate} imply
\[
cm\le\norm{[A^m,B^m]}_{\rm H}\le4\max\{a_m,b_m\}+6.
\]
Divide by $m$ and let $m\to\infty$ to obtain
$\max\{\alpha,\beta\}\ge c/4$. Choose the corresponding fixed map
$f=A$ or $f=B$. By the infimum formulas above,
\[
d_{\mathcal E}(f^m(E),E)\ge\frac c4\,m\qquad(m\ge1).
\]
Hamiltonian invariance and symmetry now give, for all $m,n\in\mathbb Z$,
\[
\frac c4|m-n|
\le d_{\mathcal E}(f^m(E),f^n(E))
\le\norm{f}_{\rm H}|m-n|.
\]
The upper bound follows by using the map $f^{m-n}$. This proves
Theorem~\ref{thm:equators}, assuming \eqref{eq-estimate}.

To prove \eqref{eq-estimate} we argue as follows. By Hamiltonian invariance, we may assume that $E=\{z=0\}$. Fix
\[
C>\max\{d_{\mathcal E}(\phi(E),E),d_{\mathcal E}(\psi(E),E)\}.
\]
There exists $\theta\in\Ham(S^2)$ with $\norm{\theta}_{\rm H}<C$
such that $\theta(E)=\phi(E)$, so $\theta^{-1}\phi(E)=E$.
Up to composition with the rotation $\mathcal R(x,y,z)=(x,-y,-z)$,
whose Hofer norm is $\norm{\mathcal R}_{\rm H}=1/2$, we can assume that
the restriction $(\phi_1)|_E$ of $\phi_1=\theta^{-1}\phi$ is isotopic
to the identity. For every $\varepsilon>0$, there exists a Hamiltonian
diffeomorphism $f_{\varepsilon}$ of norm
$\norm{f_{\varepsilon}}_{\rm H}<\varepsilon$, supported in an arbitrarily
small neighborhood of $E$, with $f_{\varepsilon}^{-1}\phi_1$ supported
in the complement of a (possibly smaller) open neighborhood of $E$.
Hence we can write $\phi_2=f_{\varepsilon}^{-1}\phi_1$ as a product
$\phi_2=\phi_+\phi_-$, where $\phi_+$ is supported in the upper (open)
hemisphere and $\phi_-$ is supported in the lower (open) hemisphere.
Of course $\phi_+$ and $\phi_-$ commute. By bi-invariance of the Hofer metric,
\[
d_{\rm H}(\phi_+\mathcal R\phi_-\mathcal R^{-1},\phi_+\phi_-)
\leq d_{\rm H}(\mathcal R\phi_-\mathcal R^{-1},\phi_-)
\leq 2\norm{\mathcal R}_{\rm H}=1.
\]
The maps $\phi_+$, $\mathcal R\phi_-\mathcal R^{-1}$ and
$\phi_3=\phi_+\mathcal R\phi_-\mathcal R^{-1}$ are all supported
in the upper hemisphere. Moreover,
\[
d_{\rm H}(\phi,\phi_3)<C+1/2+\varepsilon+1=C+3/2+\varepsilon.
\]
Similarly, there exists $\psi_3\in\Ham(S^2)$ supported in the lower
hemisphere with
\[
d_{\rm H}(\psi,\psi_3)<C+3/2+\varepsilon.
\]
Hence $\phi_3,\psi_3$ commute and
\[
\begin{aligned}
d_{\rm H}([\phi,\psi],\id)
&=d_{\rm H}([\phi,\psi],[\phi_3,\psi_3])\\
&\leq 2d_{\rm H}(\phi,\phi_3)+2d_{\rm H}(\psi,\psi_3)\\
&\leq 4(C+3/2+\varepsilon).
\end{aligned}
\]
Letting $\varepsilon\to0$ and
$C\downarrow\max\{d_{\mathcal E}(\phi(E),E),d_{\mathcal E}(\psi(E),E)\}$
gives \eqref{eq-estimate}.
\end{proof}

\medskip

\begin{proof}[Proof of Theorem \ref{thm:autonomous} assuming Theorem \ref{thm:equators}]
Recall that $\Ham(S^2)$ acts by isometries on the space of equators.
For every Hamiltonian diffeomorphism $g$ and equator $E$ define its asymptotic
displacement $$\ell(g) = \lim_{n \to +\infty} \frac{d_{\mathcal E}(g^n(E),E)}{n}\;.$$
Clearly $\ell$ does not depend on the choice of $E$, is conjugation invariant, is homogeneous in the sense that
$\ell(g^k) = |k| \ell (g)$, and $\ell$ is $1$-Lipshitz with respect to the Hofer metric.
By \cite{BFPS}, every autonomous diffeomorphism after a conjugation lies at a uniformly bounded distance from one which preserves the standard equator. Thus $\ell$ vanishes on all autonomous diffeomorphisms. At the same time  Theorem \ref{thm:equators} provides a Hamiltonian diffeomorphism $f$ with $\ell(f) > c >0$. Thus $\ell(f^n) > cn$.
It follows that the distance from $f^n$ to $\Aut(S^2)$ is $ \geq cn$,
as required.
\end{proof}


\section{Input from Floer theory}\label{sec:floer}

Consider the unit sphere carrying normalized
round area, and use coordinates $(x,y,z)$. Set
\begin{equation}\label{eq:cap-disks}
c_0=\frac1{400},\qquad
D_{\sigma,\tau}=\{\sigma x>0,\ \tau z>0,\ |xz|\ge c_0\},
\qquad\sigma,\tau\in\{\pm1\}.
\end{equation}
These four smooth closed disks have equal area $q$, and their
complement has area $\epsilon=1-4q$. Put $\eta=(q-\epsilon)/6>0$.
Label the disks $D_1,\ldots,D_4$, and put $L_i=\partial D_i$ and
$L=\bigcup_i L_i$.

We use the Lagrangian
Floer cohomology constructed in \cite[\S6]{CHMSS22} with the trivial rank-one local system
on the Lagrangian torus $T =\Sym(L)$ lying in the symmetric product $X=\Sym^4(S^2)\cong\mathbb{CP}^4$. This is Floer cohomology in
the smooth symmetric product, with Novikov exponents modified by intersection with
the diagonal $\Delta \subset X$. We recall its nonvanishing and
specify the conventions for forward Hamiltonian chords and their cappings that will
be used below. The Hamiltonian convention in \cite{CHMSS22} is opposite to ours:
a Hamiltonian $K$ here corresponds to $-K$ there, with the same Hamiltonian vector field.
We use caps from a constant reference to the chord, and take inputs at
$s=-\infty$ and outputs at $s=+\infty$. We explain below how this convention
is related by duality to the cited Floer complexes.

For clarity, write
\[
Y=X\setminus\Delta,\qquad
T=\{\{p_1,p_2,p_3,p_4\}:p_i\in L_i\}\subset Y.
\]
A configuration chord is an unordered tuple of four surface trajectories
whose initial and terminal configurations lie on $T$. The key estimate is
the following; its proof occupies this section.
\begin{lemma}[Strip Lemma]\label{lem:strip}
Let $H_0,H_1:[0,1]\times S^2\to\mathbb R$ be smooth Hamiltonians.
For every $\nu>0$, after arbitrarily small smooth perturbations needed
for transversality, there are a configuration chord for each Hamiltonian
and an actual Floer continuation strip $u$ joining them such that
\begin{equation}\label{eq:strip-main}
\eta(u\cdot\Delta)\le4\norm{H_0-H_1}_{(1,\infty)}+\nu.
\end{equation}
Both boundary components of $u$ lie on $T$, and $u\cdot\Delta\ge0$.
\end{lemma}

\subsection{The relative action and nonvanishing}
Choose a surface complex structure as in \cite[Theorem~5.11]{CHMSS22}, and use the induced
complex structure on $X$. The big diagonal
\[
 \Delta=\bigl\{\{p_1,p_2,p_3,p_4\}\in X:p_i=p_j\text{ for some }i\ne j\bigr\}
\]
is the discriminant hypersurface: under the identification with $\mathbb{CP}^4$ by
coefficients of homogeneous polynomials of degree four, it consists of polynomials
with a repeated root. Equip $X$ with a Perutz smoothing $\omega_{\mathcal V}$ of the symplectic
form induced from the product. It agrees with that form outside a small neighborhood
of $\Delta$, is compatible with the complex structure near $\Delta$, and has line area
one. The torus $T$ is disjoint from the smoothing region. The smoothing concerns the
symplectic form; $\Delta$ remains the discriminant hypersurface.

For a relative disk class $\mathfrak d\in\pi_2(X,T)$, let $\mu(\mathfrak d)$ denote its Maslov index and set
\begin{equation}\label{eq:restored-3-1}
 \omega_\eta(\mathfrak d):=\omega_{\mathcal V}(\mathfrak d)+\eta(\mathfrak d\cdot\Delta)=\frac q2\mu(\mathfrak d).
\end{equation}
The equality appears in \cite[Lemma~4.21]{CHMSS22}. Its hypothesis is satisfied because the four
disk regions have area $q$ and one boundary component, while the complementary
region has area $\epsilon$ and four boundary components, and
\[
 \epsilon+2\eta(4-1)=q.
\]
The modification in this Floer theory is in the coefficients: the Novikov weight of
a disk class $\mathfrak d$ is
\[
 \mathsf t^{\omega_\eta(\mathfrak d)}=\mathsf t^{\omega_{\mathcal V}(\mathfrak d)+\eta(\mathfrak d\cdot\Delta)}.
\]
Thus the modified area enters the Novikov exponent. The strips themselves solve the
usual Floer equation on the smooth symplectic manifold $(X,\omega_{\mathcal V})$, with almost
complex structures equal to the prescribed complex structure near $\Delta$ and
Hamiltonians spatially constant there. Since $T$ is orientable, its Maslov indices are
even. We set
\[
 \zeta=\mathsf t^q,\qquad\Lambda=\mathbb C((\zeta));
\]
recapping by a class of modified area $kq$ is recorded by $\zeta^k$. In particular, a
Maslov-index-two disk has weight $\zeta$.

Use the spin structure and orientations of \cite[Theorem~5.11]{CHMSS22}. If
$x_i\in\mathbb C^*$ is the holonomy of a rank-one local system around the $i$th
coordinate circle of $T$, the disk potential, after removing the common Novikov
factor $\zeta$, is
\begin{equation}\label{eq:restored-3-2}
 \mathcal W(x_1,x_2,x_3,x_4)=x_1+x_2+x_3+x_4+(x_1x_2x_3x_4)^{-1}.
\end{equation}
The trivial local system corresponds to $(1,1,1,1)$, which is a critical point of
$\mathcal W$. Denote by $HF_\eta(T;\Lambda)$ the self-Floer cohomology of $T$ with
this local system and these modified Novikov weights. By \cite[Lemma~6.11]{CHMSS22},
\begin{equation}\label{eq:restored-3-3}
 HF_\eta(T;\Lambda)\cong H^*(T;\Lambda)\ne0
\end{equation}
as $\Lambda$-vector spaces. The grading may be taken modulo two. This is the
nonvanishing result used in \cite[Definition~6.14]{CHMSS22} to define spectral invariants.

\subsection{Compatible Hamiltonians and forward chords}\label{sec:compatible}
For a surface Hamiltonian $H$, write
\[
 \Sym(H_t)(\{p_1,\ldots,p_4\})=\sum_{i=1}^4 H_t(p_i).
\]
A smooth Hamiltonian $K$ on $X$ is \emph{compatible} with $H$ if it agrees with
$\Sym(H)$ outside a neighborhood $\mathcal V$ of $\Delta$ disjoint from the whole moving torus
\[
 M_H=\bigcup_{t\in[0,1]}\Sym(\phi_H^t)(T),
\]
and is spatially constant on a smaller neighborhood of $\Delta$. The support of the
smoothing of the symplectic form is chosen inside this smaller neighborhood. This is
the endpoint class used in \cite[(47)--(48)]{CHMSS22}.

The chords are the actual forward trajectories
\begin{equation}\label{eq:restored-3-4}
 x(t)=\phi_K^t(x(0)),\qquad x(0),x(1)\in T.
\end{equation}

The vector fields of $K$ and $\Sym(H)$ agree near $M_H$, so uniqueness of solutions implies
that every such trajectory starting on $T$ is the corresponding unordered tuple of surface
trajectories. In particular, the cutoff creates no extra chords from $T$. Nondegeneracy means
$\phi_K^1(T)\pitchfork T$; for compatible Hamiltonians it is the transversality of
$\Sym(\phi_H^1)(T)$ and $T$.

Fix $P\in T$, independently of $K$. This point is a reference for cappings; the starting point
$x(0)$ of a chord may be any point of $T$. A cap of a chord is a map
$\widehat x:[0,1]^2\to X$, oriented by $ds\wedge dt$, with
\[
\widehat x(0,t)=P,\qquad \widehat x(1,t)=x(t),\qquad
\widehat x(s,0),\widehat x(s,1)\in T.
\]
Thus a cap is a homotopy from the constant path at $P$ to the chord, with both endpoints
allowed to move in $T$. It is not required to satisfy a Floer equation. Such caps exist since $X$
is simply connected and $T$ is connected. Identify caps of the same chord whose difference
has modified area zero and complete over $\Lambda$. With our Hamiltonian convention
and the stated orientation of caps, the ordinary action and its divisor-weighted version are
\begin{equation}
\begin{aligned}
\mathcal A_K(\widehat x)
&=\int_0^1 K_t(x(t))\,dt-\int_{[0,1]^2}\widehat x^*\omega_{\mathcal V},\\
\mathcal A_K^\eta(\widehat x)
&=\mathcal A_K(\widehat x)-\eta(\widehat x\cdot\Delta).
\end{aligned}
\label{eq:restored-3-5}
\end{equation}
The filtration below uses $\mathcal A_K^\eta$; the diagonal term records the
modified Novikov weights.
The period group is $q\mathbb Z$ by \eqref{eq:restored-3-1}. For nondegenerate $K$ the resulting complex, denoted by
$CF(K;T)$, has finite rank over $\Lambda$. We suppress $\eta$ and $\Lambda$ in this notation and write
$HF(K;T)$ for its cohomology. Multiplication by $\zeta$ lowers the action by $q$. The filtration of a chain is
the maximum action of its nonzero monomials. The differential counts rigid solutions of the
Floer equation with boundary on $T$, with the input at $s=-\infty$ and the output at $s=+\infty$;
the output cap is obtained by appending the strip to the input cap.

We relate this convention to the Hamiltonian-chord description in
\cite[\S6.2, (49)--(54)]{CHMSS22}.
Its identification with the intersection-point description is the standard Hamiltonian
change of variables
\begin{equation}
v(s,t)=\mathcal F_t(u(s,t)),\qquad \mathcal F_t=\phi_K^1\circ(\phi_K^t)^{-1}.
\label{eq:restored-3-6}
\end{equation}
It identifies a chord $x$ with the intersection point $x(1)\in\phi_K^1(T)\cap T$, and transforms the
Floer equation into the unperturbed Cauchy--Riemann equation with boundary conditions
$(\phi_K^1(T),T)$ and the transformed almost complex structures. Since $K$ is spatially constant
near $\Delta$, the maps $\mathcal F_t$ are the identity there and preserve diagonal intersection numbers. The
same change of variables applies to capped generators.
After translating the Hamiltonian sign, our caps have the opposite orientation
and our inputs and outputs are reversed relative to the cited convention.
Thus our finite-rank complex is the $\Lambda$-linear dual of the cited complex.
The Novikov powers are unchanged: if an old strip relation is
$c_{\rm out}^{\rm old}+\mathfrak d=u\#c_{\rm in}^{\rm old}$, then reversing
reference caps gives the new relation
$c_{\rm out}^{\rm new}+\mathfrak d=c_{\rm in}^{\rm new}\#u$.
Here $c_{\rm in}^{\rm new}=-c_{\rm out}^{\rm old}$ and
$c_{\rm out}^{\rm new}=-c_{\rm in}^{\rm old}$, since inputs and outputs are exchanged.
Both coefficients therefore have exponent $\omega_\eta(\mathfrak d)$;
there is no inversion of the Novikov variable. The same observation applies
to continuation maps, whose matrices are transposed. Consequently,
\cite[Lemma~6.11]{CHMSS22} and finite-dimensional duality give
\[
HF(K;T)\cong HF_\eta(T;\Lambda)^\vee\ne0,
\qquad (\cdot)^\vee=\operatorname{Hom}_\Lambda(\cdot,\Lambda).
\]
We continue to use cohomology notation, with grading modulo two.
All action comparisons below are made directly in the chord description, using the fixed
reference $P$ and \eqref{eq:restored-3-5} for every Hamiltonian.

\subsection{Continuation with a common reference}
Take two compatible, nondegenerate endpoints $K_0,K_1$, using a common symplectic form
and a common neighborhood of $\Delta$ on which both are spatially constant. Let
\begin{equation}
K^s=K_1+\beta(s)(K_0-K_1),\qquad
\beta'\ge0,\quad\beta(-\infty)=0,\quad\beta(+\infty)=1.
\label{eq:restored-3-7}
\end{equation}
The continuation map $\Phi_{01}:CF(K_1;T)\to CF(K_0;T)$ counts solutions of
\begin{equation}
\partial_su+J_{s,t}(\partial_tu-X_{K_t^s})=0,\qquad
u(s,0),u(s,1)\in T,
\label{eq:restored-3-8}
\end{equation}
with input $x_1$ at $s=-\infty$ and output $x_0$ at $s=+\infty$. All caps use the same constant
reference $P$, and the output cap is obtained by appending $u$ to the input cap. We compute
continuation directly in this model; \eqref{eq:restored-3-6} is used only for a fixed Hamiltonian.

The facts from the Floer construction that we need are the following.
\begin{enumerate}
\renewcommand{\labelenumi}{(\roman{enumi})}
\item For regular data equal to the prescribed complex structure near $\Delta$, the differentials
$d_{K_i}$ and the continuation maps are defined over $\Lambda$. For each endpoint $i=0,1$, the
Floer cohomology satisfies
\[
HF(K_i;T):=H^*(CF(K_i;T),d_{K_i})\cong HF_\eta(T;\Lambda)^\vee\ne0
\]
by \cite[Lemma~6.11]{CHMSS22} and the identification recalled in Section~\ref{sec:compatible}.
\item Continuation through Hamiltonians spatially constant on the common neighborhood of
$\Delta$ is invariant under homotopy. The reverse continuation map is inverse on cohomology.
\item The maps satisfy the weighted action estimate
\begin{equation}
\mathcal A_{K_0}^\eta(\widehat x_0)-\mathcal A_{K_1}^\eta(\widehat x_1)
\le \kappa_+-\eta r,\qquad
\kappa_+=\int_0^1\max_X(K_{0,t}-K_{1,t})\,dt,\quad r=u\cdot\Delta\ge0.
\label{eq:restored-3-9}
\end{equation}
Arbitrarily small regularizing perturbations of the homotopy add an arbitrarily small
error to $\kappa_+$.
\end{enumerate}
These follow from the continuation statements in \cite[\S6.2]{CHMSS22} by the duality
just described, with the diagonal term retained
in the energy computation of \cite[Lemma~6.15]{CHMSS22}. We now prove the action estimate
\eqref{eq:restored-3-9}.

Let $u$ be a solution of \eqref{eq:restored-3-8}, and put $r=u\cdot\Delta$. With our Hamiltonian convention, its
energy is
\[
\begin{aligned}
E(u)&=\int_{\mathbb R\times[0,1]}
\omega_{\mathcal V}(\partial_su,\partial_tu-X_{K_t^s})\,ds\,dt\\
&=\omega_{\mathcal V}(u)-\int_0^1K_{0,t}(x_0(t))\,dt
+\int_0^1K_{1,t}(x_1(t))\,dt
+\int_{\mathbb R\times[0,1]}(\partial_sK_t^s)(u)\,ds\,dt.
\end{aligned}
\]
Cap gluing therefore gives the exact identity
\begin{equation}
\mathcal A_{K_0}^\eta(\widehat x_0)-\mathcal A_{K_1}^\eta(\widehat x_1)
=-E(u)+\int_{\mathbb R\times[0,1]}(\partial_sK_t^s)(u)\,ds\,dt-\eta r.
\label{eq:restored-3-10}
\end{equation}
Near $\Delta$ the Hamiltonian vector field vanishes and the strip is holomorphic, which proves
$r\ge0$. For \eqref{eq:restored-3-7}, the middle term in \eqref{eq:restored-3-10} is at most $\kappa_+$, proving \eqref{eq:restored-3-9}.

\subsection{Key observation: strip with controlled diagonal intersection}
We prove Lemma~\ref{lem:strip}, starting with its formulation for compatible Hamiltonians on $X$.

\begin{proposition}\label{prop:selected-strip}
Let $K_0,K_1$ be compatible nondegenerate Hamiltonians as above. For
every $\nu>0$, one can choose regular continuation data such that the continuation map
$\Phi_{01}:CF(K_1;T)\to CF(K_0;T)$ counts a strip $u$ satisfying
\[
\eta(u\cdot\Delta)\le\int_0^1\operatorname{osc}_X(K_{0,t}-K_{1,t})\,dt+\nu.
\]
\end{proposition}
\begin{proof}
Put
\[
\kappa_-=\int_0^1\min_X(K_{0,t}-K_{1,t})\,dt,\qquad
\kappa_+=\int_0^1\max_X(K_{0,t}-K_{1,t})\,dt.
\]
Fix $0<\delta<\nu/4$. Choose regular continuation maps
\[
\Phi_{01}:CF(K_1;T)\longrightarrow CF(K_0;T),\qquad
\Phi_{10}:CF(K_0;T)\longrightarrow CF(K_1;T)
\]
whose regularization errors are less than $\delta$. Thus every strip $u$ counted by $\Phi_{01}$ satisfies
\[
\mathcal A_{K_0}^\eta(\widehat x_0)-\mathcal A_{K_1}^\eta(\widehat x_1)
\le \kappa_++\delta-\eta(u\cdot\Delta),
\]
where the output cap is obtained by appending $u$ to the input cap. The reverse map raises the
filtration by at most $-\kappa_-+\delta$. These maps induce inverse isomorphisms on cohomology.

Choose $0\ne\alpha_1\in HF(K_1;T)$ and set $\alpha_0=(\Phi_{01})_*(\alpha_1)\ne0$.
Write $\lambda_{K_i}(\xi)$ for the filtration level of a chain $\xi\in CF(K_i;T)$, and let
\[
C_i=\inf\{\lambda_{K_i}(\xi):d_{K_i}\xi=0,\ [\xi]=\alpha_i\},\qquad i=0,1.
\]
These numbers are finite: each complex has finite rank over the complete discretely valued
field $\Lambda$, and its boundary subspace is closed. The filtration bounds for the two continuation
maps give
\[
C_0\le C_1+\kappa_++\delta,\qquad C_1\le C_0-\kappa_-+\delta.
\]
Consequently,
\begin{equation}
\kappa_--\delta\le C_0-C_1\le \kappa_++\delta.
\label{eq:restored-3-11}
\end{equation}
By the definition of $C_1$, choose a cycle $\xi_1$ representing $\alpha_1$ such that
$\lambda_{K_1}(\xi_1)\le C_1+\delta$. Its continuation image $\xi_0=\Phi_{01}(\xi_1)$ represents
$\alpha_0$, so $\lambda_{K_0}(\xi_0)\ge C_0$. The maximum defining $\lambda_{K_0}(\xi_0)$
is attained. Hence $\xi_0$ contains a capped-chord monomial $\widehat x_0$ with nonzero coefficient and
\[
\mathcal A_{K_0}^\eta(\widehat x_0)\ge C_0.
\]
We now choose a strip counted in this coefficient. If an input monomial $\widehat x_1$ contributes to
it, the continuation estimate and $u\cdot\Delta\ge0$ imply
\[
\mathcal A_{K_1}^\eta(\widehat x_1)
\ge\mathcal A_{K_0}^\eta(\widehat x_0)-\kappa_+-\delta.
\]
The Novikov condition implies that only finitely many monomials of $\xi_1$ satisfy this lower
bound. For each such input monomial, the coefficient of $\widehat x_0$ in its continuation image is a
finite signed count of strips. Since the coefficient of $\widehat x_0$ in $\xi_0$ is nonzero, there is an input
monomial $\widehat x_1$ occurring in $\xi_1$ and a strip $u$ counted by $\Phi_{01}$ joining these capped chords.

For this input monomial, $\mathcal A_{K_1}^\eta(\widehat x_1)\le C_1+\delta$.
Applying \eqref{eq:restored-3-9}, with regularization error bounded by $\delta$, to the chosen strip gives
\[
\begin{aligned}
\eta(u\cdot\Delta)
&\le \kappa_++\delta+\mathcal A_{K_1}^\eta(\widehat x_1)-\mathcal A_{K_0}^\eta(\widehat x_0)\\
&\le \kappa_++C_1-C_0+2\delta\\
&\le \kappa_+-\kappa_-+3\delta\\
&<\int_0^1\operatorname{osc}_X(K_{0,t}-K_{1,t})\,dt+\nu.
\end{aligned}
\]
This proves the proposition.
\end{proof}

Here is the resulting estimate for surface Hamiltonians. Given $H_0,H_1$, first make
arbitrarily small surface perturbations if needed to achieve chord nondegeneracy.

Recall from Section~\ref{sec:compatible} that
\[
M_{H_i}=\bigcup_{t\in[0,1]}\Sym(\phi_{H_i}^t)(T),\qquad i=0,1.
\]
These compact sets are disjoint from $\Delta$. Choose neighborhoods $\overline{\mathcal U}\subset \mathcal V$ of $\Delta$ disjoint from
$M_{H_0}\cup M_{H_1}$, with the smoothing supported in $\mathcal U$.

Let $\chi$ be zero near $\overline{\mathcal U}$ and one outside $\mathcal V$. Set
\begin{equation}
K_{i,t}=\chi\Sym(H_{i,t})+(1-\chi)4\overline H_i(t),\qquad
\overline H_i(t)=\int_{S^2}H_{i,t}\,\omega.
\label{eq:restored-3-12}
\end{equation}
These Hamiltonians are compatible and have a common constant region near $\Delta$. Their
difference is a convex combination of $\Sym(H_0-H_1)$ and four times the surface average.
Consequently
\begin{equation}
\int_0^1\operatorname{osc}_X(K_{0,t}-K_{1,t})\,dt
\le4\norm{H_0-H_1}_{(1,\infty)}.
\label{eq:restored-3-13}
\end{equation}
Proposition~\ref{prop:selected-strip} supplies an actual strip between physical surface chords with
\begin{equation}
\eta r\le4\norm{H_0-H_1}_{(1,\infty)}+\nu.
\label{eq:restored-3-14}
\end{equation}
For degenerate endpoints we use strips for the perturbed surface paths and pass only
numerical inequalities to the limit. In particular, no convergence to a strip for degenerate
data is required. This proves Lemma~\ref{lem:strip}.

\section{Braids, half twists, and the comparison estimate}\label{sec:braid-topology}
A crucial feature of our approach is as follows: we view Hamiltonian trajectories with the endpoints
on the symmetric product of our circles $L_i$, $i=1,2,3,4$ as a braid on $S^2$. Roughly speaking, the Floer-theoretic
Lemma ~\ref{lem:strip} enables us to compare algebraically such a braids corresponding to different Hamiltonians.

Let $Y=X\setminus\Delta$ be the unordered configuration space of four points, and fix
$P=\{p_1,p_2,p_3,p_4\}\in T$, with $p_i\in L_i$.
The inclusion $\pi_1(T,P)\to\pi_1(Y,P)$ is zero: each coordinate
generator contracts through its disk while avoiding the other three points. Consequently, a
path $x:[0,1]\to Y$ with endpoints on $T$ determines a braid
\[
\beta_x=[\gamma_0\mathbin\sharp x\mathbin\sharp\gamma_1]\in\pi_1(Y,P)=B_4(S^2),
\]
where $\gamma_0,\gamma_1$ are paths in $T$ from $P$ to $x(0)$ and from $x(1)$ to $P$, respectively. This braid is
independent of the closing paths and invariant under homotopies of $x$ in $Y$ with endpoints
on $T$. Here $\sharp$ denotes concatenation.

The evaluation map
\[
\operatorname{ev}:B_4(S^2)\longrightarrow\Mod(S^2,P),\qquad
\operatorname{ev}([\beta])=[\Psi_1],
\]
is defined by an ambient isotopy satisfying $\Psi_0=\mathrm{id}$ and $\Psi_t(P)=\beta(t)$.
It is well-defined and is an antihomomorphism with our convention.

For a Hamiltonian chord $x(t)=\phi_t(x(0))$, the mapping class
$\operatorname{ev}(\beta_x)$ depends only on $\phi_1$ and $x(0)$:
changing the isotopy inserts a braid induced by a loop of ambient diffeomorphisms, whose
evaluation is the identity.

Recall that a half twist exchanges two marked points by a half turn along an embedded
arc joining them and disjoint from the other two points. Small meridians of the smooth
part of $\Delta$ represent half twists.

\begin{lemma}\label{lem:meridian}
Let $K_0,K_1$ satisfy the hypotheses of Proposition~\ref{prop:selected-strip}.
For every $\nu>0$, there
is a strip $u$ counted by a continuation map from $CF(K_1;T)$ to $CF(K_0;T)$, with asymptotic
chords $x_0,x_1$, such that, writing $r=u\cdot\Delta$,
\[
\eta r\le\int_0^1\operatorname{osc}_X(K_{0,t}-K_{1,t})\,dt+\nu,
\]
and $\beta_{x_1}\beta_{x_0}^{-1}$ is a product of $r$ conjugates of inverse half twists.
\end{lemma}

\begin{proof}
Choose the strip supplied by Proposition~\ref{prop:selected-strip}. Its intersections with $\Delta$ have positive
multiplicities. Resolve them into $r$ simple intersections with the smooth part of $\Delta$, and
remove small disks around their preimages. With orientation $ds\wedge dt$ and
output at $s=+\infty$, the positively oriented outer boundary represents
$\beta_{x_0}\beta_{x_1}^{-1}$, up to conjugation. It is a product of $r$
conjugates of positive meridians. Inverting this boundary relation gives the
asserted factorization of $\beta_{x_1}\beta_{x_0}^{-1}$ into conjugates of inverse
half twists. The closing paths in $T$ cause no ambiguity because
$\pi_1(T)\to\pi_1(Y)$ is zero. This factorization argument applies to every
continuation strip.
\end{proof}

For the direct proof of Equators Theorem \ref{thm:equators} in Section \ref{sec:direct-equators}
we need the following consequence of this lemma. Given a homogeneous quasimorphism
$\overline Q:\Mod(S^2,P)\to\mathbb R$ of defect at most $D$, put
\[
Q=\overline Q\circ\operatorname{ev}:B_4(S^2)\to\mathbb R.
\]
Then $Q$ is also a homogeneous quasimorphism of defect at most $D$.

\begin{corollary} If a homogeneous quasimorphism $Q$ vanishes on half twists, then
\begin{equation}
|Q(\beta_{x_0})-Q(\beta_{x_1})|
\le Dr\le\frac D\eta\left(\int_0^1\operatorname{osc}_X(K_{0,t}-K_{1,t})\,dt+\nu\right).
\label{eq:restored-4-1}
\end{equation}
\end{corollary}

\begin{proof}
Homogeneity implies that $Q$ also vanishes on inverse half twists.
For $r\ge1$, conjugation invariance and the quasimorphism inequality give
\[
|Q(\beta_{x_1}\beta_{x_0}^{-1})|\le(r-1)D.
\]
One further application of the quasimorphism inequality gives \eqref{eq:restored-4-1}. If $r=0$, the two braids
coincide.
\end{proof}

If $\phi$ fixes $T$ pointwise, the chords of any fixed isotopy from $\mathrm{id}$ to $\phi$ form a continuous
family parametrized by $T$, and hence determine the same braid $\beta_\phi$. Its evaluation is the
mapping class $[\phi]_P\in\Mod(S^2,P)$. For the cap-fixing maps constructed above, moving the
marked points into the disks identifies this class with the relative mapping class in \eqref{eq:torus-quotient}.
Homogeneity gives
\begin{equation}
Q(\beta_{\phi^n})=nQ(\beta_\phi),\qquad n\in\mathbb Z.
\label{eq:restored-4-2}
\end{equation}
Although $\beta_\phi$ may depend on the chosen isotopy, $Q(\beta_\phi)$ does not.

We also record the stability needed for degenerate endpoints. For a fixed smooth surface
Hamiltonian $H$, every chord of a sufficiently small smooth perturbation of $H$ determines the
same braid as some chord of $H$. Indeed, otherwise a sequence of counterexamples has, by
compactness of $T$, a subsequence converging to an $H$-chord. Sufficiently close configuration
paths are homotopic in $Y$ with endpoints on $T$, a contradiction. Thus any common value or
uniform bound for $Q$ on the chords of $H$ persists under sufficiently small perturbations.

Finally, let $\phi,\psi\in\Ham(S^2)$. Fix a Hamiltonian path $G$ from $\mathrm{id}$ to $\psi$, and choose a
Hamiltonian path $F$ from $\mathrm{id}$ to $\psi^{-1}\phi$ whose Hofer length is arbitrarily close to $d_{\rm H}(\phi,\psi)$.
The composed path is generated by
\begin{equation}
H_t=G_t+F_t\circ(\phi_G^t)^{-1},\qquad
\phi_H^1=\phi,\quad
\norm{H-G}_{(1,\infty)}=\norm{F}_{(1,\infty)}.
\label{eq:restored-4-3}
\end{equation}
Apply Lemma~\ref{lem:meridian} to compatible Hamiltonians obtained from $H,G$ by the common cutoff
\eqref{eq:restored-3-12}. Using \eqref{eq:restored-3-13}, and allowing arbitrarily small surface perturbations to ensure nondegeneracy,
we obtain the following: for every $\varepsilon>0$, the choices can be made so that there
are chords $x_0,x_1$ of the perturbed paths satisfying
\begin{equation}
|Q(\beta_{x_0})-Q(\beta_{x_1})|
\le\frac{4D}\eta d_{\rm H}(\phi,\psi)+\varepsilon.
\label{eq:restored-4-4}
\end{equation}
In the applications below, the preceding stability observation allows us to use the values or
bounds for the unperturbed chords and then let $\varepsilon\to0$.

\section{Starting the construction}

\subsection{The four caps and the smooth twist maps}\label{sec:construction}
As above, we work on the round sphere with coordinates $x,y,z$.  Since a coordinate strip $\{|x|<s\}$ has area $s$, the inclusion
\[
\{|xz|<c_0\}\subset\{|x|<\sqrt{c_0}\}\cup\{|z|<\sqrt{c_0}\}
\]
gives
\begin{equation}\label{eq:cap-constants}
\epsilon\le\frac1{10},\qquad
q=\frac{1-\epsilon}4\ge\frac9{40},\qquad
\eta:=\frac{q-\epsilon}6=\frac{5q-1}6\ge\frac1{48}.
\end{equation}
Choose a smooth nondecreasing function $\rho:[-1,1]\to[0,1]$ with
\[
\rho(s)=0\ (s\le-\delta),\qquad
\rho(s)=1\ (s\ge\delta),\qquad
\rho(s)+\rho(-s)=1,\qquad\delta=\frac1{800}.
\]
For an axis coordinate $s$ (here $s$ is $x$,$y$, or $z$) use an angular coordinate $\theta_s$ with
$\omega=(4\pi)^{-1}d\theta_s\wedge ds$. The Hamiltonian
\begin{equation}\label{eq:twist-hamiltonian}
h_{\rm tw}(s)=-\frac12\int_{-1}^{s}\rho(u)\,du
\end{equation}
satisfies $\dot\theta_s=-4\pi h_{\rm tw}'(s)=2\pi\rho(s)$ with our sign convention,
and hence generates $\theta_s(t)=\theta_s(0)+2\pi t\rho(s)$. Let $a$ and $b$ be the
time-one maps for $s=x$ and $s=z$, respectively. These are same-sign twists
about the equators $x=0$ and $z=0$. Their time-one maps are supported in the
corresponding $\delta$-strips, and hence fix neighborhoods of all four disks
pointwise. Moreover,
\begin{equation}\label{eq:explicit-word}
\norm{a}_{\rm H},\norm{b}_{\rm H}\le\frac12,\qquad
f_0:=a^2b^{-1}a^3b^{-2},\qquad\norm{f_0}_{\rm H}\le4.
\end{equation}
Take $E_0=\{x=0\}$. Each component of $S^2\setminus E_0$ contains exactly two disks.

Mark one point in each disk, and denote the resulting four-point set
by $P$. Here $\Mod(S^2,4)=\Mod(S^2,P)$ denotes the group of
orientation-preserving mapping classes preserving $P$ setwise;
permutations of the marked points are allowed.

The double cover $q:T^2\to S^2$ branched over $P$ is a torus.
Identify $T^2$ with $\mathbb R^2/\mathbb Z^2$, so that the deck
involution is $\iota(x)=-x$. Every homeomorphism representing
$g\in\Mod(S^2,4)$ has two lifts to $T^2$, differing by $\iota$.
Their actions on $H_1(T^2;\mathbb Z)$ are therefore matrices
$A$ and $-A$. Consequently, after choosing a homology basis,
\begin{equation}\label{eq:torus-quotient}
\pi:\Mod(S^2,4)\longrightarrow\PSL(2,\mathbb Z)
=\langle S,R\mid S^2=R^3=1\rangle,
\qquad
\pi(g)=[A],
\end{equation}
is a well-defined homomorphism. It is surjective: every linear
torus automorphism commutes with $\iota$ and hence descends to
the marked sphere. Its kernel is
\[
K\cong(\mathbb Z/2\mathbb Z)^2,
\]
represented by the descents of the four translations
$x\mapsto x+\xi$ with $2\xi=0$ in $T^2$.
See \cite[Corollary~3.3 and Remark~3.4]{AMU06}.

We explain explicitly how this quotient detects reducibility.
An essential simple closed curve $\gamma\subset S^2\setminus P$
separates $P$ into two pairs. Its inverse image under $q$
consists of two parallel essential simple closed curves on $T^2$.
The primitive homology class of either component, taken up to
sign, determines $[\gamma]$ uniquely. In fact, this gives a
bijection
\[
\left\{
\begin{array}{c}
\text{isotopy classes of essential}\\
\text{unoriented simple closed curves in }S^2\setminus P
\end{array}
\right\}
\longleftrightarrow
\{v\in\mathbb Z^2:\ v\text{ is primitive}\}/\{\pm1\};
\]
see \cite[Proposition~2.6]{FM12}.
This bijection is equivariant: if $[\gamma]$ corresponds to
$\{\pm v\}$ and $\pi(g)=[A]$, then $g[\gamma]$ corresponds to
$\{\pm Av\}$.

Suppose now that $\pi(g)$ is parabolic. Then
$|\operatorname{tr}A|=2$ and $A\ne\pm I$, so there is a primitive
integer vector $v$ satisfying $Av=\pm v$. Equivariance and
injectivity of the above correspondence imply that $g$ preserves
the isotopy class of the associated essential curve $\gamma$.
Moreover, $g$ has infinite order because $\pi(g)$ does.
Thus $g$ is reducible.

If $\pi(g)$ is elliptic, it has finite order $m$, so
$g^m\in K$. Since every element of $K$ has square equal to the
identity, $g^{2m}=1$. Thus $g$ is periodic.
The same conclusion holds when $\pi(g)=1$.

Finally, if $\pi(g)$ is hyperbolic, then $g$ has infinite order
and no positive power can preserve an essential curve:
otherwise equivariance would give $A^kv=\pm v$ for some
$k\ge1$ and primitive $v$, which is impossible for a hyperbolic
matrix. Hence $g$ is pseudo-Anosov by the Nielsen--Thurston
classification. This compatibility of types is also recorded
in \cite[Lemma~3.6]{AMU06}.

With compatible choices of twist sign and marking, the images of $a,b$ are
\[
U=\begin{pmatrix}1&2\\0&1\end{pmatrix}=(SR)^2,\qquad
V=\begin{pmatrix}1&0\\-2&1\end{pmatrix},\qquad V^{-1}=(SR^2)^2.
\]
The factor two comes from the two parallel lifts of a separating curve.
The image of $f_0$ is
\begin{equation}\label{eq:word-matrix}
W=U^2V^{-1}U^3V^{-2}
=\begin{pmatrix}241&58\\54&13\end{pmatrix}
=(SR)^4(SR^2)^2(SR)^6(SR^2)^4.
\end{equation}

Since $\det W=1$ and $\operatorname{tr}W=254>2$, the matrix $W$ is hyperbolic,
and no positive power of $W$ has eigenvalue $1$ or $-1$. The relative mapping
class of $f_0$ on the four-punctured sphere is therefore neither periodic nor
reducible. Indeed, periodicity would imply $W^k=\pm I$ for some $k\ge1$.
Reducibility would imply that some positive power preserves an essential
simple closed curve. Its lifts to the branched double cover are parallel
essential curves on the torus, whose primitive homology class $v$ would
satisfy $W^kv=\pm v$, again a contradiction. The Nielsen--Thurston
classification now implies that the relative mapping class of $f_0$ is
pseudo-Anosov.

Here the pseudo-Anosov assertion concerns the mapping class relative to the
four marked disks; $f_0$ itself is a smooth Hamiltonian representative, equal
to the identity near those disks.

The construction is related to the use of pseudo-Anosov mapping classes by
Jia and Webb \cite{JW25}. They construct quasimorphisms vanishing on stabilizers
of balanced curves and prove infinite diameter of the space of equators
for quantitative fragmentation metrics. Our Hofer estimate uses configuration
braids together with a Floer continuation strip. The connection between
Lagrangian configurations and Hofer geometry was developed in \cite{PS23}; the
Floer theory used here is the link Floer theory of \cite{CHMSS22}.

\section{The free group embedding}
\label{sec:f2embedding}

The construction has two ingredients. We choose two fixed Hamiltonian
maps, which automatically define a homomorphism from the free group.
We then prove a lower bound by decoding their relative mapping classes
and applying a single continuation strip. No realization of the relations
in a larger mapping-class group is required.

\subsection{The generators and the upper bound}

Write $F_2=\langle x,y\rangle$ and recall its intrinsic generator norm
\[
\nu_F(w)=\min\left\{r:w=\prod_{i=1}^r c_i e_i c_i^{-1},\quad
 c_i\in F_2,\quad e_i\in\{x,x^{-1},y,y^{-1}\}\right\}.
\]
Set $d_F(v,w)=\nu_F(v^{-1}w)$. The empty product has length zero.
Reversing, concatenating, and conjugating
factorizations show that $\nu_F$ is symmetric, subadditive, and conjugation
invariant. Only the identity has norm zero, so $d_F$ is a bi-invariant
metric. Conjugation is allowed by elements of $F_2$ itself. This norm is
different from ordinary reduced-word length and must not be identified
without proof with a norm inherited from a larger group.

Let $a,b$ be the Hamiltonian twists defined in Section \ref{sec:construction}, and set
\begin{equation}\label{eq:f2-Hamiltonian-generators}
A=a^3b^{-1}a^3,\qquad B=b^{-3}a^2b^{-3},\qquad
\Phi(x)=A,\quad\Phi(y)=B.
\end{equation}
The universal property of the free group defines a homomorphism
$\Phi:F_2\to\Ham(S^2)$. All its values fix the four caps pointwise.
The Hofer upper bound follows at once from
\[
\norm{A}_{\rm H}\le\tfrac72,\qquad \norm{B}_{\rm H}\le4,
\qquad\text{hence}\qquad \norm{\Phi(w)}_{\rm H}\le4\nu_F(w).
\]
Indeed, each allowed conjugate in a shortest factorization of $w$ maps
to a Hamiltonian conjugate of $A^{\pm1}$ or $B^{\pm1}$.

\subsection{The relative classes and the decoder}

For the lower bound, put
\[
G=C_2*C_3=\langle S,R\mid S^2=R^3=1\rangle,
\qquad U=(SR)^2,\quad V^{-1}=(SR^2)^2.
\]
The projected relative mapping classes of $a,b$ are $U,V$, respectively.
Thus the projected homomorphism of \eqref{eq:f2-Hamiltonian-generators} is
\begin{equation}\label{eq:f2-coding-homomorphism}
\iota:F_2\longrightarrow G,\qquad
\iota(x)=U^3V^{-1}U^3,\qquad
\iota(y)=V^{-3}U^2V^{-3}.
\end{equation}
We will prove, simultaneously, that $\iota$ is injective and that
\begin{equation}\label{eq:f2-algebraic-lower}
\nu_F(w)\le104\,\tau(\iota(w))\qquad(w\in F_2).
\end{equation}
Here $\tau$ is the conjugation-invariant torsion norm on $G$:
\[
\tau(g)=\min\left\{r:g=\prod_{j=1}^{r}c_j e_j c_j^{-1},\quad
c_j\in G,\quad e_j\in\{S,R,R^2\}\right\}.
\]
The empty product has length zero. The conjugating elements range over
all of $G$ and do not contribute to the length. In particular,
$\tau(SR)\le2$, and the same bound holds for every conjugate of
$(SR)^{\pm1}$.

For a finite string $c=c_1\cdots c_k$ with $c_i\in\{1,2\}$, define
\[
\mathsf u(c)=SR^{c_1}\cdots SR^{c_k}.
\]
In an exponent string, $1^j$ or $2^j$ means $j$ consecutive copies of
that digit. Equation~\eqref{eq:f2-coding-homomorphism} becomes
\[
\iota(x)=\mathsf u(1^6\,2^2\,1^6),\qquad
\iota(y)=\mathsf u(2^6\,1^4\,2^6).
\]
A \emph{marker word} is a fixed short syllable pattern used to recognize
a source letter inside its coded image. Choose
\begin{equation}\label{eq:f2-markers}
p_x=\mathsf u(1^2\,2^2\,1^3),\qquad
p_y=\mathsf u(2\,1^4\,2^2).
\end{equation}
Each has syllable length $\ell_{\rm mark}=14$, where $R^2$ is one syllable.
The four words $p_x,p_y,p_x^{-1},p_y^{-1}$ are pairwise distinct.
A \emph{marker occurrence} is a particular interval of fourteen
consecutive syllable positions whose reading is one of these words.
For example,
\[
\iota(x)=(SR)^4\,\underbrace{(SR)^2(SR^2)^2(SR)^3}_{p_x}\,(SR)^3,
\]
so $p_x$ occurs in the interior of the block encoding $x$.

Define $\mathcal D:G\to F_2$ as follows. Scan the unique reduced syllable
word for $g$ from left to right. Record \emph{every} occurrence, including
overlapping occurrences, of $p_x,p_y,p_x^{-1},p_y^{-1}$ as
$x,y,x^{-1},y^{-1}$, respectively. List these labels in order of their
starting positions and interpret their product as an element of $F_2$.
There is at most one label at any starting position, since the four
patterns are distinct and have the same length. Set $\mathcal D(1)=1$.
This definition makes no homomorphism assertion about $\mathcal D$.

\subsection{Why every source word can be recovered}

\begin{lemma}[Exact decoding]\label{lem:f2-exact}
For every $w\in F_2$, one has $\mathcal D(\iota(w))=w$.
In particular, $\iota$ is injective.
\end{lemma}
\begin{proof}
Write $w$ as a freely reduced word in $x^{\pm1},y^{\pm1}$, and substitute
the four corresponding reduced syllable blocks in $G$. Drop the
$S$-syllables temporarily and record the exponents of the $R$-syllables.
A \emph{run} is a maximal consecutive string of one digit. The four
oriented blocks have exponent strings
\[
\begin{array}{c|c|c}
\text{source letter}&\text{exponent string of its image}&\text{central run}\\ \hline
x&1^6\,2^2\,1^6&2^2\\
x^{-1}&2^6\,1^2\,2^6&1^2\\
y&2^6\,1^4\,2^6&1^4\\
y^{-1}&1^6\,2^4\,1^6&2^4
\end{array}
\]
Positive blocks begin with $S$ and end in the $C_3$-factor; negative
blocks begin in the $C_3$-factor and end with $S$.

We describe all reductions at a block junction. For two positive blocks,
or two negative blocks, the syllable junction is already reduced. In the
exponent string, the exterior runs either remain separate runs of length
six or merge into one run of length twelve.

At a mixed-sign junction, the source letters involve different generators:
otherwise they would cancel in the freely reduced word $w$. For a
positive-negative junction, the adjacent $C_3$-syllables are equal, so
their product is nonidentity and fuses by $RR=R^2$ or $R^2R^2=R$.
For a negative-positive junction, the two adjacent copies of $S$ cancel,
and then the same kind of nontrivial $C_3$-fusion occurs. In the exponent
string either case leaves a flank run of length five on each side of
one new run of length one. The fused syllable has digit opposite to
the two neighboring flank runs.

These reductions at different junctions do not interact: the two
junctions adjacent to a block affect different length-six exterior runs,
and leave its central run untouched. The resulting syllable word is
reduced. Its noncentral maximal exponent runs have lengths in
$\{1,5,6,12\}$. Therefore its runs of length two or four are exactly
the original central runs, in source-letter order.

The markers detect exactly these four central types. Their exponent
strings, with the full syllable orientation retained, are
\[
\begin{array}{c|c|c}
\text{marker}&\text{exponent string}&\text{central run}\\ \hline
p_x&1^2\,2^2\,1^3&2^2\\
p_x^{-1}&2^3\,1^2\,2^2&1^2\\
p_y&2\,1^4\,2^2&1^4\\
p_y^{-1}&1^2\,2^4\,1&2^4
\end{array}
\]
The inverse markers begin in $C_3$ and end with $S$, as required by
syllable inversion; they are not obtained by applying $\mathsf u$ to
the displayed strings. Each marker requires a whole central run, bounded
on both sides by the other digit. It therefore cannot match a proper
subinterval of a longer run. Conversely, every surviving central run
has flanks of length at least five, whereas every marker requires at
most three flank digits on either side. It gives exactly one occurrence
of its designated marker: the central run fixes the start position of
the exponent window, and the alternating syllables fix the specified
initial factor of the marker. There are no other occurrences.

The output labels are consequently exactly the letters of $w$, in their
original order. This proves the identity, also for the empty word.
If $\iota(w)=1$, applying $\mathcal D$ now gives $w=1$.
\end{proof}

\subsection{Uniform control of the decoder}
The next lemma concerns the conjugation-invariant metric $d_F$, rather
than ordinary word distance. This distinction makes it possible to
move local errors through arbitrarily long words at no additional cost.

\begin{lemma}\label{lem:f2-decoder}
The decoder satisfies
\begin{align}
\mathcal D(g^{-1})&=\mathcal D(g)^{-1},
\label{eq:f2-decoder-inverse}\\
d_F\bigl(\mathcal D(gh),\mathcal D(g)\mathcal D(h)\bigr)&\le78,
\label{eq:f2-decoder-defect}\\
\nu_F(\mathcal D(g))&\le104\,\tau(g)
\label{eq:f2-decoder-norm}
\end{align}
for all $g,h\in G$.
\end{lemma}
\begin{proof}
Inverting a reduced syllable word reverses the order of all occurrences
of patterns of a fixed length and replaces each by its inverse pattern.
This proves \eqref{eq:f2-decoder-inverse}.

Put $J=\ell_{\rm mark}-1=13$. If $uv$ is a reduced concatenation, an occurrence
lies wholly in $u$, wholly in $v$, or crosses their junction. At most
$J$ possible starting positions give the third type. Equal pattern
lengths ensure that all crossing occurrences come after occurrences
wholly in $u$ and before occurrences wholly in $v$. Thus, in $F_2$,
\[
\mathcal D(uv)=\mathcal D(u)\,E\,\mathcal D(v),
\qquad |E|_{\mathrm{ordinary}}\le J.
\]
As $\nu_F(E)\le |E|_{\mathrm{ordinary}}$ and $d_F$ is bi-invariant,
\begin{equation}\label{eq:f2-reduced-join}
d_F\bigl(\mathcal D(uv),\mathcal D(u)\mathcal D(v)\bigr)\le J.
\end{equation}
Indeed, the difference of the two products is a conjugate of $E$ or
$E^{-1}$, and conjugation does not change its $\nu_F$-norm.

For a general product $gh$, remove its maximal inverse suffix-prefix
cancellation and write the reduced concatenations
\[
g=uc,\qquad h=c^{-1}v.
\]
If $uv$ is reduced, apply \eqref{eq:f2-reduced-join} to these two
expressions and to $uv$. The contributions of $\mathcal D(c)$ and
$\mathcal D(c^{-1})$ cancel by \eqref{eq:f2-decoder-inverse}, so the
total error is at most $3J$.

Otherwise $u=u'\gamma$ and $v=\delta v'$, where $\gamma,\delta$ are
syllables in the same factor and $\gamma\delta\ne1$. Then
\[
g=u'\gamma c,\qquad h=c^{-1}\delta v',\qquad
gh=u'(\gamma\delta)v'
\]
are reduced concatenations. One-syllable words contain no marker.
Using at most two reduced joins in each of these three expressions
compares their decoder values, respectively, with
\[
\mathcal D(u')\mathcal D(c),\qquad
\mathcal D(c)^{-1}\mathcal D(v'),\qquad
\mathcal D(u')\mathcal D(v').
\]
Each comparison costs at most $2J$. Bi-invariance permits multiplication
of these comparisons without changing the errors, so the total error
is at most $6J=78$. Empty constituent words cause no difficulty and
only decrease the number of joins. This proves
\eqref{eq:f2-decoder-defect}.

Let $e\in\{S,R,R^2\}$. There exists a conjugate $cec^{-1}$ which is already in reduced form.
Indeed, choose an initial reduced conjugator $c$. If its last syllable lies in the same cyclic subgroup
$\langle S\rangle$ or $\langle R\rangle$ as $e$, remove that
syllable: it commutes with $e$, so this does not change the group element
$cec^{-1}$. The resulting conjugator, which we still call $c,$ is either empty or ends
in the other cyclic subgroup. Thus the concatenation
$c\,e\,c^{-1}$ is reduced, with at most two joins. Hence, as $\mathcal D(e)=1,$


\[
\nu_F\bigl(\mathcal D(cec^{-1})\bigr)\le2J=26.
\]
For $g\ne1$, choose a factorization into $r=\tau(g)$ such conjugates.
Repeated application of \eqref{eq:f2-decoder-defect} and the triangle
inequality gives
\[
\nu_F(\mathcal D(g))
\le26r+78(r-1)\le104r.
\]
For $g=1$, both sides of \eqref{eq:f2-decoder-norm} vanish. This proves
the final assertion.
\end{proof}

Combining Lemmas~\ref{lem:f2-exact} and~\ref{lem:f2-decoder} gives
\[
\nu_F(w)=\nu_F\bigl(\mathcal D(\iota(w))\bigr)
\le104\,\tau(\iota(w)),
\]
which is \eqref{eq:f2-algebraic-lower}. The decoder retains an ordered
product of labels in a free group. It is not a homogeneous real
quasimorphism, so the bound on real homogeneous quasimorphisms on
single commutators does not obstruct this argument.

\subsection{The direct continuation-strip estimate}

\begin{proof}[Completion of the proof of Theorem~\ref{thm:f2embedding}]
Fix $w\in F_2$. Every $\Phi(w)$ fixes the four caps pointwise, and its
projected relative mapping class is $\iota(w)$. Let $\varepsilon>0$
and choose a Hamiltonian path to $\Phi(w)$ of length at most
$\norm{\Phi(w)}_{\rm H}+\varepsilon$. Compare it with the constant path.
After sufficiently small regularizing perturbations, the strip lemma
(Lemma~\ref{lem:strip}) supplies an actual continuation strip with
$r$ diagonal intersections such that
\begin{equation}\label{eq:f2-strip-budget}
\eta r\le4\bigl(\norm{\Phi(w)}_{\rm H}+\varepsilon\bigr)+o(1).
\end{equation}
The error tends to zero with the regularizing perturbations.

All source chord labels have projected mapping class $\iota(w)$,
and all identity chord labels have trivial projected class. This remains
true under the sufficiently small perturbations used above. The
meridian relation (Lemma~\ref{lem:meridian}) therefore expresses
$\iota(w)$, or its inverse according to the orientation convention, as
a product of $r$ conjugates of projected half twists or their inverses.
Each such factor is conjugate to $(SR)^{\pm1}$ and has $\tau$-norm at most two.
Consequently,
\begin{equation}\label{eq:f2-meridian-budget}
\tau(\iota(w))\le2r.
\end{equation}
The decoder estimate now gives
\[
\nu_F(w)
=\nu_F\bigl(\mathcal D(\iota(w))\bigr)
\le104\tau(\iota(w))\le208r.
\]
Combining this with \eqref{eq:f2-strip-budget}, then sending the
regularization error and $\varepsilon$ to zero, yields
\[
\norm{\Phi(w)}_{\rm H}\ge\frac\eta{832}\nu_F(w)
\ge\frac1{39936}\nu_F(w),
\]
since $\eta\ge1/48$. Together with the upper bound already proved,
\begin{equation}\label{eq:f2-final-norm}
\frac1{39936}\nu_F(w)\le\norm{\Phi(w)}_{\rm H}\le4\nu_F(w).
\end{equation}

In particular, $\Phi(w)=\mathrm{id}$ forces $w=1$, so $\Phi$ is
injective. For $v,w\in F_2$, the homomorphism property and
bi-invariance of Hofer's metric give
\[
d_{\rm H}\bigl(\Phi(v),\Phi(w)\bigr)
=\norm{\Phi(v^{-1}w)}_{\rm H}.
\]
Apply \eqref{eq:f2-final-norm} to $v^{-1}w$ to obtain the claimed
bi-Lipschitz estimates. The comparison uses the ambient Hofer norm:
the nearly shortest Hamiltonian path was arbitrary and was not required
to preserve the caps. The maps $A,B$ are fixed smooth Hamiltonian
maps, independent of the source words.
\end{proof}

\begin{remark}[Why the original pair of twists is insufficient]
The direct homomorphism $x\mapsto a$, $y\mapsto b$ does not satisfy a
quasi-isometric lower bound for $\nu_F$. The half-turn
$\mathcal S(x,y,z)=(z,-y,x)$ satisfies $\mathcal S a\mathcal S^{-1}=b$ and
$\norm{\mathcal S}_{\rm H}\le1/2$. Consequently
\[
\norm{a^n b^{-n}}_{\rm H}=\norm{[a^n,\mathcal S]}_{\rm H}\le2\norm{\mathcal S}_{\rm H}\le1
\qquad(n\ge1).
\]
On the other hand, abelianization $F_2\to\mathbb Z^2$ gives
$\nu_F(x^ny^{-n})\ge2n$: every allowed conjugate maps to a signed
coordinate vector, whereas $(n,-n)$ has $\ell^1$-norm $2n$.
The reverse inequality is given by the written word, so
$\nu_F(x^ny^{-n})=2n$. The generator words in
\eqref{eq:f2-Hamiltonian-generators}, together with their exact decoder,
remove this collapse of the intrinsic norm.
\end{remark}

\section{A direct proof of the equator theorem}
\label{sec:direct-equators}

This section gives an independent route to Theorem~\ref{thm:equators}.
It uses the continuation strip and its braid factorization, but neither
the free-group embedding nor the deductions through commutators and
hemisphere-supported maps. It also identifies one explicit map whose
cyclic orbit of equators escapes linearly.

Retain the twists $a,b$ and the four caps from
Section~\ref{sec:construction}, and put
\[
f_0=a^2b^{-1}a^3b^{-2},\qquad E_0=\{x=0\}.
\]
The equator $E_0$ has two caps on each side. The map $f_0$ fixes the caps
pointwise, has $\norm{f_0}_{\rm H}\le4$, and has projected relative mapping class
\[
W=U^2V^{-1}U^3V^{-2}
 =(SR)^4(SR^2)^2(SR)^6(SR^2)^4.
\]
We will prove
\begin{equation}\label{eq:direct-equator-escape}
\frac1{142848}|m-n|
\le d_{\mathcal E}\bigl(f_0^m(E_0),f_0^n(E_0)\bigr)
\le4|m-n|\qquad(m,n\in\mathbb Z).
\end{equation}
The role of the counting quasimorphism below is to distinguish the braid
of $f_0^n$ from every braid arising from a map that preserves $E_0$.

\subsection{A counting quasimorphism adapted to the explicit word}

A homogeneous quasimorphism on a group is a function $\mathfrak q$ such that
$\mathfrak q(g^n)=n\mathfrak q(g)$ for all integers $n$, and
\[
|\mathfrak q(gh)-\mathfrak q(g)-\mathfrak q(h)|\le D
\]
for a fixed constant $D$. It is conjugation invariant.

 We need a
quasimorphism on
$G=C_2*C_3=\langle S,R\mid S^2=R^3=1\rangle$ with the following
properties.

\begin{lemma}\label{lem:counting-q}
There is a homogeneous quasimorphism $\mathfrak q:G\to\mathbb R$ with defect at
most $D=744$ such that
\[
\mathfrak q(W)=1,
\]
and $\mathfrak q$ vanishes on every parabolic and every finite-order element.
\end{lemma}

\begin{proof}
Let $w_0$ be the displayed reduced syllable word for $W$. Its length is
$32$, where $R^2$ counts as one syllable. For a reduced word $g$, let
$N_{w_0}(g)$ count all contiguous occurrences of $w_0$, including
overlapping occurrences, and define
\[
C(g)=N_{w_0}(g)-N_{w_0^{-1}}(g).
\]
At a reduced concatenation, only occurrences crossing the join affect
additivity. There are at most $31$ possible starting positions, and at
each position the signed contribution is $-1$, $0$, or $1$. Thus the
additivity error at one reduced join is at most $31$.

For an arbitrary product, remove the maximal inverse suffix--prefix
cancellation and write $g=uc$, $h=c^{-1}v$. The contributions of $c$
and $c^{-1}$ cancel, since $C(c^{-1})=-C(c)$. If $uv$ is reduced,
three reduced joins compare $C(gh)$ with $C(g)+C(h)$. Otherwise,
write $u=u'\alpha$ and $v=\beta v'$, with $\alpha,\beta$ in the same
factor and $\alpha\beta\ne1$. Then
\[
g=u'\alpha c,\qquad h=c^{-1}\beta v',\qquad
gh=u'(\alpha\beta)v'
\]
are reduced concatenations. At most two joins in each of these three
expressions suffice, and one-syllable words have zero count. Therefore
\[
|C(gh)-C(g)-C(h)|\le6\cdot31=186.
\]

Set
\[
\mathfrak q(g)=\lim_{k\to\infty}\frac{C(g^k)}{k}.
\]
The limit exists by the elementary almost-additive-sequence argument:
adding $186$ makes $C(g^k)$ subadditive, and subtracting $186$ makes it
superadditive; their normalized limits coincide. Iterating the defect
inequality also gives $|\mathfrak q(g)-C(g)|\le186$. It follows that $\mathfrak q$ has
defect at most $4\cdot186=744$. The construction gives homogeneity;
conjugation changes $C(g^k)$ by a bounded amount, independent of $k$,
so $\mathfrak q$ is conjugation invariant.

The cyclic exponent code of $w_0$ is
\[
1111\ 22\ 111111\ 2222.
\]
It contains ten $1$'s and six $2$'s, whereas its inverse has the two
counts exchanged. Thus no cyclic shift of $w_0$ is $w_0^{-1}$.
Moreover, this cyclic word is primitive. A repetition count would
divide both ten and six and could only be two; but a square would have
equal lengths for its two cyclic runs of $1$'s, whereas here these
lengths are four and six. Hence an occurrence of $w_0$ in $W^k$ starts
exactly at a period start, and no occurrence of $w_0^{-1}$ is possible.
Consequently
\[
C(W^k)=k,\qquad \mathfrak q(W)=1.
\]
Every parabolic in $G$ is conjugate to a nonzero power of $SR$ or its
inverse. The cyclic exponent strings of these words use only one
symbol, so the signed count and its homogenization vanish on them.
Finally, homogeneity forces $\mathfrak q$ to vanish on finite-order elements.
\end{proof}

Let $\pi:\Mod(S^2,P)\to G$ be the map obtained from the branched torus
cover, and let $\operatorname{ev}$ be the evaluation of configuration
braids from Section~\ref{sec:braid-topology}. Define
\[
Q=\mathfrak q\circ\pi\circ\operatorname{ev}:B_4(S^2)\longrightarrow\mathbb R.
\]
Our chronological concatenation convention makes $\operatorname{ev}$
an antihomomorphism. This does not affect the real quasimorphism
inequality, homogeneity, or the defect bound.

A reducible mapping class of the four-punctured sphere preserves an
essential separating curve. Its lift to the branched torus preserves
a rational slope. The stabilizer of a rational slope in $G$ consists
of parabolic elements and the identity. Thus $Q$ vanishes on reducible
mapping classes. It also vanishes on periodic classes, whose images in
$G$ have finite order. In particular, it vanishes on every half twist.

\subsection{The strip comparison and the two types of endpoint}

Recall the configuration torus $T=\operatorname{Sym}(L)$ and the braid
label $\beta_x$ of a configuration chord $x$ with endpoints on $T$.
The label is defined by closing the chord with paths in $T$; these
choices cause no ambiguity because $\pi_1(T)\to\pi_1(Y)$ is zero.

Suppose a continuation strip has endpoint labels $\beta_0,\beta_1$ and
$r$ diagonal intersections. Lemma~\ref{lem:meridian} expresses
$\beta_1\beta_0^{-1}$ as a product of $r$ conjugates of inverse half twists.
Their $Q$-values vanish by homogeneity. If $r\ge1$, conjugation invariance
and the defect inequality give
\[
|Q(\beta_1\beta_0^{-1})|\le(r-1)D,
\qquad |Q(\beta_0)-Q(\beta_1)|\le rD.
\]
If $r=0$, the labels coincide, and the last inequality still holds.
Together with Lemma~\ref{lem:strip}, this gives, for a selected pair of
chords of paths ending at $\phi,\psi$,
\begin{equation}\label{eq:direct-Q-comparison}
|Q(\beta_0)-Q(\beta_1)|
\le\frac{4D}{\eta}d_{\rm H}(\phi,\psi)+\varepsilon,
\end{equation}
where $\varepsilon>0$ can be arbitrarily small. The paths can be chosen
with Hamiltonian difference of length arbitrarily close to the endpoint
Hofer distance, as explained in Section~\ref{sec:braid-topology}.
The strip selects its chords. We therefore need statements valid for
\emph{all} possible endpoint chords.

For a path ending at $f_0^n$, every configuration chord has
\begin{equation}\label{eq:direct-source-Q}
Q(\beta_x)=n.
\end{equation}
Indeed, $f_0^n$ fixes $T$ pointwise. Its chords vary continuously with
their starting configuration in the connected torus $T$, so their
labels have one common evaluated mapping class. This class projects to
$W^n$, and $\mathfrak q(W^n)=n$. A different Hamiltonian path changes the braid
by the evaluation of a loop of ambient diffeomorphisms; its mapping
class is the identity, so the value is unchanged.

\begin{lemma}\label{lem:direct-equator-stabilizer}
Every configuration chord of a Hamiltonian path ending at a map
$h\in\operatorname{Stab}(E_0)$ has $Q$-value zero. The stabilizer here
is setwise, and its elements may interchange the hemispheres.
\end{lemma}

\begin{proof}
Close the chord to the base configuration $P\in T$ by two paths in
$T$. Realize these closing paths by ambient isotopies in collars of
the four link components, with endpoints $\Gamma_0,\Gamma_1$. These collars are
disjoint from $E_0$, so the two maps preserve $E_0$. The evaluated
mapping class of the closed chord is represented by $\Gamma_1h\Gamma_0$.
This map preserves both $P$ and $E_0$.

There are two points of $P$ on each side of $E_0$, so $E_0$ is an
essential curve in the four-punctured sphere. Its branched lift determines
a rational slope preserved by the projected mapping class. That projection
is therefore parabolic or the identity, and its $Q$-value is zero.
Preserving the curve setwise suffices; no preservation of its two sides is needed.
\end{proof}

The small regularizing perturbations required in Floer theory do not
change these endpoint conclusions. For a fixed Hamiltonian path, every
chord of a sufficiently small smooth perturbation has the same braid
label as some exact chord. Otherwise, compactness of the initial
configuration torus supplies a sequence converging to an exact chord,
and sufficiently close paths in the collision-free configuration space
are homotopic with endpoints on $T$. This is the stability argument
from Section~\ref{sec:braid-topology}. It applies separately to each
chosen pair of paths; no uniform perturbation size is needed.

\subsection{Passing from the stabilizer to the equator space}

Apply \eqref{eq:direct-Q-comparison} to $\phi=f_0^n$ and any
$h\in\operatorname{Stab}(E_0)$. By
\eqref{eq:direct-source-Q} and
Lemma~\ref{lem:direct-equator-stabilizer}, every selected pair of
endpoint chords has $Q$-difference of absolute value $|n|$. Sending
the perturbation and length errors to zero gives
\[
d_{\rm H}(f_0^n,h)\ge\frac{\eta}{4D}|n|.
\]

The quotient formula for the Lagrangian Hofer metric is
\begin{equation}\label{eq:direct-quotient-formula}
d_{\mathcal E}(E_0,\phi(E_0))
=\inf_{h\in\operatorname{Stab}(E_0)}d_{\rm H}(\phi,h).
\end{equation}
Indeed, all maps carrying $E_0$ to $\phi(E_0)$ form the coset
$\phi\operatorname{Stab}(E_0)$. The definition of $d_{\mathcal E}$,
inversion of the stabilizer, and bi-invariance of Hofer's metric
identify the two infima in \eqref{eq:direct-quotient-formula}.
Therefore
\[
d_{\mathcal E}(E_0,f_0^n(E_0))\ge\frac{\eta}{4D}|n|.
\]
For the upper bound, use $\norm{f_0}_{\rm H}\le4$. Hamiltonian invariance of
$d_{\mathcal E}$ now yields
\[
\frac{\eta}{4D}|m-n|
\le d_{\mathcal E}\bigl(f_0^m(E_0),f_0^n(E_0)\bigr)
\le4|m-n|.
\]
Finally, $\eta\ge1/48$ and $D=744$, so
$\eta/(4D)\ge1/142848$. This proves
\eqref{eq:direct-equator-escape}, and hence gives the promised
independent proof of Theorem~\ref{thm:equators}.

\section{Miscellaneous results}\label{sec:misc}

We present a few further corollaries of the results in this paper.

\subsection{On the asymptotic cone of $\Ham(S^2).$}\label{sec:cone}

Recall that the asymptotic cone of a group $G$ with a bi-invariant metric is naturally a group \cite{CZ11}.

\begin{corollary}\label{cor:hofer-asymptotic-cone}
For every non-principal ultrafilter $\omega$, the asymptotic cone
\[
\operatorname{Cone}_{\omega}
\bigl(\operatorname{Ham}(S^2),d_{\rm H}\bigr),
\]
defined with scaling sequence $s_n=n$, contains a subgroup
isomorphic to $F_2$ and a nontrivial element of order two.
\end{corollary}

\begin{proof}
Let $\nu_F$ denote the conjugation-invariant generator norm on
$F_2=\langle x,y\rangle$. The quasi-isometric group monomorphism
\[
\Phi:(F_2,\nu_F)\longrightarrow
\bigl(\operatorname{Ham}(S^2),d_{\rm H}\bigr)
\]
constructed above induces a bi-Lipschitz group monomorphism
\[
\Phi_\omega:
\operatorname{Cone}_\omega(F_2,\nu_F)
\longrightarrow
\operatorname{Cone}_\omega
\bigl(\operatorname{Ham}(S^2),d_{\rm H}\bigr),
\qquad
[g_n]\longmapsto[\Phi(g_n)].
\]
Indeed, the additive constants in the quasi-isometry inequalities
vanish after division by $n$ and passage to the ultralimit.

By \cite[Proposition~3.21 and Lemma~3.29(a)]{KL26},
the homomorphisms
$\delta_n:F_2\to F_2$ defined by
$\delta_n(x)=x^n$ and $\delta_n(y)=y^n$ satisfy
\[
\nu_F(\delta_n(w))=n\nu_F(w).
\]
Consequently, $w\mapsto[\delta_n(w)]$ is an isometric group
monomorphism from $F_2$ into its asymptotic cone.
Furthermore, \cite[Theorem~1.1]{KLN26} shows that this cone
contains a nontrivial element of order two.
Both assertions therefore follow by applying $\Phi_\omega$.
\end{proof}

\subsection{On maps of bounded topological entropy}

For every fixed $C>0$, one can replace the autonomous maps in
Theorem~\ref{thm:autonomous} by Hamiltonian diffeomorphisms of
topological entropy $h_{\mathrm{top}}$ less than $C$, after replacing $f$ by a
sufficiently large positive iterate if necessary.
Indeed, Humili\`ere \cite[Proposition~3.1.10]{Hum17} proved that
\[
\ell(g)\le K_\omega h_{\mathrm{top}}(g),
\]
where $K_\omega>0$ depends only on the area normalization.
Since the asymptotic displacement $\ell$ is $1$-Lipschitz with respect to Hofer's metric and satisfies
$\ell(f^n)=n\ell(f)$, we obtain
\[
d_{\rm H}\bigl(f^n,
 \{g\in\operatorname{Ham}(S^2):h_{\mathrm{top}}(g)<C\}\bigr)
\ge n\ell(f)-K_\omega C.
\]
As $\ell(f)>0$, this gives the required linear growth.

\subsection{Large-scale geometry of the space of equators}

\begin{corollary}\label{cor:equator-quasi-flats}
Let $\Phi:F_2\to\operatorname{Ham}(S^2)$ be the homomorphism
from Theorem \ref{thm:f2embedding}. The orbit map
\[
\Psi:F_2\longrightarrow\mathcal E,
\qquad
\Psi(w)=\Phi(w)(E_0),
\]
is an injective quasi-isometric embedding, where $F_2$ carries
its conjugation-invariant generator metric. More precisely,
\begin{equation}\label{eq:equator-orbit-qi}
\frac{\nu_F(v^{-1}w)-234}{39936}
\le
d_{\mathcal E}\bigl(\Psi(v),\Psi(w)\bigr)
\le
4\nu_F(v^{-1}w)
\qquad (v,w\in F_2).
\end{equation}
Moreover, for every $m\ge1$, there exists a quasi-isometric
embedding
\[
(\mathbb Z^m,\|\cdot\|_1)
\longrightarrow
(\mathcal E,d_{\mathcal E})
\] of $\mathbb Z^m$ with the $L^1$-norm.
\end{corollary}

\begin{proof}
Retain the notation $G=C_2*C_3$, the homomorphism
$\iota:F_2\to G$, the decoder $\mathcal D:G\to F_2$,
and the torsion norm $\tau$.

We first observe that
\begin{equation}\label{eq:decoder-parabolic-bound}
\nu_F(\mathcal D(p))\le156
\end{equation}
whenever $p\in G$ is parabolic or trivial.
Indeed, write $p=c(SR)^kc^{-1}$ for some $k\in\mathbb Z$.
The reduced word for $(SR)^k$ contains no marker, so
$\mathcal D((SR)^k)=1$.
Two applications of the decoder defect estimate, together
with $\mathcal D(c^{-1})=\mathcal D(c)^{-1}$, give
\eqref{eq:decoder-parabolic-bound}.

Fix $w\in F_2$ and $h\in\operatorname{Stab}(E_0)$.
For every $\varepsilon>0$, the strip comparison and the
stabilizer argument provide a parabolic or trivial element
$p\in G$ and an integer $r\ge0$ such that
\[
\eta r\le4d_{\rm H}(\Phi(w),h)+\varepsilon,
\qquad
\tau\bigl(\iota(w)^{-1}p\bigr)\le2r.
\]
Here the second inequality follows because every projected
half twist or its inverse has $\tau$-norm at most two.
The decoder estimates imply
\begin{align*}
\nu_F(w)
&\le
\nu_F\bigl(\mathcal D(\iota(w)^{-1}p)\bigr)
+\nu_F(\mathcal D(p))+78\\
&\le
104\,\tau\bigl(\iota(w)^{-1}p\bigr)+234\\
&\le208r+234.
\end{align*}
Using $\eta\ge1/48$ and letting $\varepsilon\to0$, we obtain
\[
\nu_F(w)\le39936\,d_{\rm H}(\Phi(w),h)+234.
\]
Taking the infimum over $h\in\operatorname{Stab}(E_0)$ gives
\[
d_{\mathcal E}(E_0,\Phi(w)(E_0))
\ge\frac{\nu_F(w)-234}{39936}.
\]
The upper bound follows from
\[
d_{\mathcal E}(E_0,\Phi(w)(E_0))
\le\|\Phi(w)\|_{\rm H}\le4\nu_F(w).
\]
Hamiltonian invariance now yields
\eqref{eq:equator-orbit-qi}.

To see that $\Psi$ is injective, suppose that
$\Psi(v)=\Psi(w)$ and put $u=v^{-1}w$.
Then $\Phi(u)$ preserves $E_0$, so
\eqref{eq:equator-orbit-qi} gives
$\nu_F(u^n)\le234$ for every $n\ge1$.
Every nontrivial element of a free group is detected by a
homogeneous quasimorphism and therefore has unbounded powers
in the conjugation-invariant generator norm; see
\cite{BGKM16}.
Hence $u=1$.

Finally, by \cite[Section~3]{BGKM16}, for every $m\ge1$
there is a quasi-isometric embedding of
$(\mathbb Z^m,\|\cdot\|_1)$ into $(F_2,d_F)$.
Composing with $\Psi$ proves the quasi-flat assertion.
\end{proof}

\subsection{On random walks in $\Ham(S^2)$}

We prove the following result, which complements Dawid's square-root lower bounds for the expected
Hofer norm of random walks~\cite{Daw26} by exhibiting a finitely
supported symmetric random walk on $\Ham(S^2)$ with positive linear
escape rate in Hofer's metric.

\begin{corollary}[Linear escape of symmetric random walks]
\label{cor:random-walks}
Let $A=\Phi(x)$ and $B=\Phi(y)$ be the Hamiltonian
diffeomorphisms constructed in Theorem \ref{thm:f2embedding} above.
Let $X_1,X_2,\ldots$ be independent random variables, each
uniformly distributed on $\{A,A^{-1},B,B^{-1}\}$, and put
\[
Z_n=X_1\cdots X_n.
\]
There exist constants
$v_{\rm H}\ge v_{\mathcal E}>0$ such that
\[
\frac{\|Z_n\|_{\rm H}}{n}\longrightarrow v_{\rm H},
\qquad
\frac{d_{\mathcal E}(Z_n(E),E)}{n}
\longrightarrow v_{\mathcal E}
\]
almost surely and in $L^1$, for every equator $E$.
In particular,
\[
\mathbb E\|Z_n\|_{\rm H}=v_{\rm H}n+o(n),
\qquad
\mathbb E\,d_{\mathcal E}(Z_n(E),E)
=v_{\mathcal E}n+o(n).
\]
\end{corollary}

\begin{proof}
Let $\xi_1,\xi_2,\ldots$ be independent and uniformly
distributed on $\{x,x^{-1},y,y^{-1}\}$, and write
\[
W_n=\xi_1\cdots\xi_n,
\qquad
X_j=\Phi(\xi_j),
\qquad
Z_n=\Phi(W_n).
\]
By Gadgil--Krishnapur \cite[Theorem~1]{GK23},
\begin{equation}\label{eq:random-cancellation-drift}
\lim_{n\to\infty}\frac{\mathbb E\,\nu_F(W_n)}{n}
=\lambda_2
\end{equation}
for some $\lambda_2>0$.

Each of the three sequences
\[
\nu_F(W_n),\qquad
\|Z_n\|_{\rm H},\qquad
d_{\mathcal E}(Z_n(E_0),E_0)
\]
is subadditive over the Bernoulli shift $\sigma$:
if $U_n$ denotes any one of these quantities, then
\[
U_{n+m}\le U_n+U_m\circ\sigma^n.
\]
For the last sequence, this follows from the triangle
inequality and Hamiltonian invariance of $d_{\mathcal E}$.
The shift is ergodic, and all three quantities lie between
$0$ and $4n$.
Kingman's subadditive ergodic theorem \cite{Kin73}
therefore gives deterministic almost-sure limits after
division by $n$.
The uniform bounds imply convergence in $L^1$ as well.
By \eqref{eq:random-cancellation-drift}, the limit for
$\nu_F(W_n)/n$ is $\lambda_2$.

Corollary~\ref{cor:equator-quasi-flats} and the Hofer
upper bound for $\Phi$ give
\[
\frac{\nu_F(W_n)-234}{39936}
\le
d_{\mathcal E}(Z_n(E_0),E_0)
\le
\|Z_n\|_{\rm H}
\le
4\nu_F(W_n).
\]
Passing to the limits yields
\[
0<
\frac{\lambda_2}{39936}
\le v_{\mathcal E}
\le v_{\rm H}
\le4\lambda_2.
\]

Finally, for every equator $E$, Hamiltonian invariance
and the triangle inequality imply
\[
\left|
d_{\mathcal E}(Z_n(E),E)
-d_{\mathcal E}(Z_n(E_0),E_0)
\right|
\le2d_{\mathcal E}(E,E_0).
\]
Thus the same limit holds for every $E$.
The expectation asymptotics follow from the
$L^1$ convergence.
\end{proof}


\begin{thebibliography}{CHMSS22}


\bibitem[AGK+19]{AGK19}
D.~Alvarez-Gavela, V.~Kaminker, A.~Kislev, K.~Kliakhandler,
A.~Pavlichenko, L.~Rigolli, D.~Rosen, O.~Shabtai,
B.~Stevenson and J.~Zhang,
\emph{Embeddings of free groups into asymptotic cones of
Hamiltonian diffeomorphisms},
J. Topol. Anal. \textbf{11} (2019), no.~2, 467--498.


\bibitem[AM24]{AM24}
M.~R.~R.~Alves and M.~Meiwes,
\emph{Braid stability and the Hofer metric},
Ann. H. Lebesgue \textbf{7} (2024), 521--581.

\bibitem[AMU06]{AMU06}
J.~E.~Andersen, G.~Masbaum, and K.~Ueno,
Topological quantum field theory and the Nielsen--Thurston
classification of $M(0,4)$,
\emph{Math. Proc. Cambridge Philos. Soc.}
\textbf{141} (2006), no.~3, 477--488.

\bibitem[BGKM16]{BGKM16}
M.~Brandenbursky, \'{S}.~R.~Gal, J.~K\k{e}dra and M.~Marcinkowski,
\emph{The cancellation norm and the geometry of bi-invariant
word metrics},
Glasgow Math. J. \textbf{58} (2016), no.~1, 153--176.

\bibitem[BMS22]{BMS22}
M.~Brandenbursky, M.~Marcinkowski and E.~Shelukhin,
\emph{The Schwarz--Milnor lemma for braids and area-preserving
diffeomorphisms},
Selecta Math. (N.S.) \textbf{28} (2022), no.~4,
Paper No.~74, 20~pp.

\bibitem[BS15]{BS15}
M.~Brandenbursky and E.~Shelukhin,
\emph{On the $L^p$-geometry of autonomous Hamiltonian diffeomorphisms
of surfaces},
Math. Res. Lett. \textbf{22} (2015), no.~5, 1275--1294.

\bibitem[BS17]{BS17}
M.~Brandenbursky and E.~Shelukhin,
\emph{The $L^p$-diameter of the group of area-preserving
diffeomorphisms of $S^2$},
Geom. Topol. \textbf{21} (2017), no.~6, 3785--3810.

\bibitem[BS26]{BS26}
M.~Brandenbursky and E.~Shelukhin,
\emph{The $L^p$-diameter of the space of contractible loops},
Bull. Lond. Math. Soc. \textbf{58} (2026), no.~5, e70396.
arXiv:2509.07270.


\bibitem[BFPS]{BFPS} Buhovsky L, Feuerstein B, Polterovich L, Shelukhin E. A dichotomy for the Hofer growth of area preserving maps on the sphere via symmetrization. arXiv preprint arXiv:2408.08854. 2024 Aug 16.

\bibitem[Che24]{Che24}
G.~Chen,
\emph{Stability of the braid types defined by the symplectomorphisms
preserving a link},
J. Fixed Point Theory Appl. \textbf{26} (2024), no.~1, Paper No.~8.
arXiv:2309.03845.

\bibitem[Cho24]{Cho24}
A.~Chor,
\emph{Eggbeater dynamics on symplectic surfaces of genus 2 and 3},
Ann. Math. Qu\'ebec \textbf{48} (2024), no.~1, 113--142.

\bibitem[CG24]{CG24}
D.~Connery-Grigg,
\emph{Hamiltonian Floer theory on surfaces: linking, positively
transverse foliations and spectral invariants},
Invent. Math. \textbf{237} (2024), no.~3, 1377--1468.

\bibitem[CHMSS22]{CHMSS22}
D. Cristofaro-Gardiner, V. Humili\`ere, C. Y. Mak, S. Seyfaddini and I. Smith,
\emph{Quantitative Heegaard Floer cohomology and the Calabi invariant},
Forum Math. Pi \textbf{10} (2022), e27.
doi:10.1017/fmp.2022.18. Numbered references in the text follow the revised
version, arXiv:2105.11026v2 (2024).

\bibitem[CZ11]{CZ11}
D.~Calegari and D.~Zhuang,
\emph{Stable $W$-length},
in \emph{Topology and geometry in dimension three},
Contemp. Math., vol.~560,
Amer. Math. Soc., Providence, RI, 2011, pp.~145--169.


\bibitem[Daw25]{Daw25}
A.~Dawid,
\emph{Random Hamiltonians I: Probability measures and random walks
on the Hamiltonian diffeomorphism group},
preprint (2025), arXiv:2510.03190.

\bibitem[Daw26]{Daw26}
A.~Dawid,
\emph{Random Hamiltonians II: A central limit theorem and the Hofer
geometry of random walks},
preprint (2026), arXiv:2608.07673.


\bibitem[FM12]{FM12}
B.~Farb and D.~Margalit,
\emph{A Primer on Mapping Class Groups},
Princeton Mathematical Series, vol.~49,
Princeton University Press, 2012.


\bibitem[GK23]{GK23}
S.~Gadgil and M.~Krishnapur,
Random words in free groups, non-crossing matchings and
RNA secondary structures,
Indian J. Pure Appl. Math. \textbf{54} (2023),
no.~1, 146--158.
doi:10.1007/s13226-022-00240-x.

\bibitem[Gru26]{Gru26}
N.~Grunder,
\emph{Braid stability and a Floer theory of braid isotopies},
preprint (2026), arXiv:2609.34961.

\bibitem[Hof90]{Hof90}
H. Hofer, \emph{On the topological properties of symplectic maps},
Proc. Roy. Soc. Edinburgh Sect. A \textbf{115} (1990), 25--38.
\bibitem[JW25]{JW25}
Y. Jia and R. Webb, \emph{Balanced curves, quasimorphisms, and the equator
conjecture}, preprint (2025), arXiv:2511.22698.

\bibitem[Hum17]{Hum17}
V.~Humili\`ere,
\emph{G\'eom\'etrie symplectique $C^0$ et s\'electeurs d'action},
Habilitation thesis, Universit\'e Pierre et Marie Curie, 2017.
\url{https://webusers.imj-prg.fr/~vincent.humiliere/papers/HDR.pdf}.




\bibitem[Kha09]{Kha09}
M.~Khanevsky,
\emph{Hofer's metric on the space of diameters},
J. Topol. Anal. \textbf{1} (2009), no.~4, 407--416.
arXiv:0907.1499.

\bibitem[Kha11]{Kha11}
M.~Khanevsky,
\emph{Geometric and topological aspects of Lagrangian submanifolds:
intersections, diameter and Floer theory},
Ph.D. thesis, Tel Aviv University, 2011.

\bibitem[Kha16]{Kha16}
M.~Khanevsky,
\emph{Hamiltonian commutators with large Hofer norm},
J. Symplectic Geom. \textbf{14} (2016), no.~4, 1175--1188.

\bibitem[Kha21]{Kha21}
M.~Khanevsky,
\emph{Non-autonomous curves on surfaces},
J. Mod. Dyn. \textbf{17} (2021), 305--317.
arXiv:1906.07884.

\bibitem[Kha22a]{Kha22a}
M.~Khanevsky,
\emph{A gap in the Hofer metric between integrable and autonomous
Hamiltonian diffeomorphisms on surfaces},
preprint (2022), arXiv:2205.03492.

\bibitem[Kha22b]{Kha22b}
M.~Khanevsky,
\emph{Hofer's distance between eggbeaters and autonomous Hamiltonian
diffeomorphisms on surfaces},
preprint (2022), arXiv:2205.15806.

\bibitem[Kin73]{Kin73}
J.~F.~C.~Kingman,
Subadditive ergodic theory,
Ann. Probab. \textbf{1} (1973), no.~6, 883--899.

\bibitem[KL26]{KL26}
J.~K\k{e}dra and A.~Libman,
\emph{Directional asymptotic cones of groups equipped with
bi-invariant metrics},
arXiv:2207.12704, version~3 (2026).

\bibitem[KLN26]{KLN26}
J.~K\k{e}dra, A.~Libman, and Z.~Nie,
\emph{Torsion in asymptotic cones of free groups},
arXiv:2608.10854 (2026).


\bibitem[McD]{McD}
D.~McDuff,
\emph{Problem session},
Oberwolfach Rep. \textbf{3} (2006), no.~3, 2049--2051
(Report~33/2006).

\bibitem[MS17]{MS17}
D.~McDuff and D.~Salamon,
\emph{Introduction to symplectic topology},
third edition, Oxford University Press, Oxford, 2017.

\bibitem[Mor25]{Mor25}
F.~Morabito,
\emph{Link Floer homology and a Hofer pseudometric on braids},
J. Symplectic Geom. \textbf{23} (2025), no.~3, 713--749.
arXiv:2301.08105.

\bibitem[MT25]{MT25}
F.~Morabito and I.~Trifa,
\emph{Hofer energy and link preserving diffeomorphisms in higher genus},
J. Mod. Dyn. \textbf{21} (2025), 543--577.
arXiv:2404.01052.

\bibitem[Pol01]{Pol01}
L. Polterovich, \emph{The geometry of the group of symplectic diffeomorphisms},
Lectures in Mathematics ETH Z\"urich (Birkh\"auser, Basel, 2001).

\bibitem[PS16]{PS16}
L.~Polterovich and E.~Shelukhin,
\emph{Autonomous Hamiltonian flows, Hofer's geometry and persistence
modules},
Selecta Math. (N.S.) \textbf{22} (2016), no.~1, 227--296.

\bibitem[PS23]{PS23}
L. Polterovich and E. Shelukhin, \emph{Lagrangian configurations and Hamiltonian
maps}, Compos. Math. \textbf{159} (2023), 2483--2520.
doi:10.1112/S0010437X23007455.

\bibitem[PSS17]{PSS17}
L.~Polterovich, E.~Shelukhin and V.~Stojisavljevi\'c,
\emph{Persistence modules with operators in Morse and Floer theory},
Mosc. Math. J. \textbf{17} (2017), no.~4, 757--786.


\bibitem[ST26]{ST26}
B.~Serraille, I.~Trifa,
\emph{On link quasimorphisms on the sphere and the equator
              conjecture},
Math. Ann. \textbf{395} (2026), no.~4, Paper No.~94, 22.

\bibitem[Tri24]{Tri24}
I.~Trifa,
\emph{Hofer distance on Lagrangian links inside the disc},
J. Fixed Point Theory Appl. \textbf{26} (2024), no.~4, Paper No.~59.
arXiv:2409.17828.

\bibitem[Wan26]{Wan26}
Z.~W.~Wang,
\emph{Complexity of Hofer's geometry in higher-dimensional manifolds},
J. Mod. Dyn. \textbf{22} (2026), 489--515.

\bibitem[Zha19]{Zha19}
J.~Zhang,
\emph{$p$-cyclic persistent homology and Hofer distance},
J. Symplectic Geom. \textbf{17} (2019), no.~3, 857--927.

\end{thebibliography}
\end{document}